\documentclass[12pt]{amsart} 
\usepackage{amssymb,amsmath,amscd,mathrsfs} 
\usepackage{a4wide}
\usepackage{verbatim}
\usepackage[all]{xy} 
\usepackage{color}

\renewcommand{\theequation}{\arabic{section}.\arabic{equation}}
\newtheorem{Question}{Question}
\newtheorem{theorem}{Theorem}[section]
\newtheorem{lemma}[theorem]{Lemma}
\newtheorem{proposition}[theorem]{Proposition}
\newtheorem{corollary}[theorem]{Corollary}

\theoremstyle{definition} 
\newtheorem{definition}[theorem]{Definition}
\newtheorem{example}[theorem]{Example}
\theoremstyle{remark}
\newtheorem{remark}[theorem]{Remark}

\newcommand{\cC}{\mathcal{C}}
\newcommand{\cI}{\mathcal{I}}
\newcommand{\cJ}{\mathcal{J}}

\newcommand{\cU}{\mathcal{U}}

\newcommand{\cN}{\mathcal{N}}
\newcommand{\cS}{\mathcal{S}}

\newcommand{\cV}{\mathcal{V}}

\newcommand{\Dr}{\mathscr{D}}
\newcommand{\Fr}{\mathscr{F}}
\newcommand{\Gr}{\mathscr{G}}
\newcommand{\Hr}{\mathscr{H}}

\newcommand{\Kr}{\mathscr{K}}
\newcommand{\Lr}{\mathscr{L}}

\newcommand{\Sr}{\mathscr{S}}
\newcommand{\Bc}{\mathcal{B}}
\newcommand{\Dc}{\mathcal{D}}
\newcommand{\Kc}{\mathcal{K}}

\newcommand{\ccV}{{}^c\cV}

\newcommand{\ctX}{{}^cT\oX}
\newcommand{\ctsX}{{}^cT^*\oX}
\newcommand{\cun}{\cC^{\infty}}
\newcommand{\cunc}{\cC^{\infty}_c}
\newcommand{\cz}{\mathbb{C}}
\newcommand{\Diff}{\mathrm{Diff}}
\newcommand{\Diffc}{\mathrm{Diff}_c}

\newcommand{\ess}{\mathrm{ess}}

\newcommand{\im}{\mathrm{Im}}

\newcommand{\oX}{\overline{X}}
\newcommand{\px}{\partial_x}
\newcommand{\rz}{\mathbb{R}}
\newcommand{\supp}{\mathrm{supp}}
\newcommand{\C}{\mathbb{C}}
\newcommand{\R}{\mathbb{R}}
\newcommand{\N}{\mathbb{N}}
\newcommand{\qz}{\mathbb{Q}}

\newcommand{\siges}{\sigma_\ess}

\newcommand{\Tr}{\mathrm{Tr}}

\newcommand{\vol}{\operatorname{vol}}
\newcommand{\Vol}{\operatorname{Vol}}

\newcommand{\zz}{\mathbb{Z}}
 
\def\cchi{\raisebox{.45 ex}{$\chi$}}
\def\build#1_#2^#3{\mathrel{\mathop{\kern 0pt#1}\limits_{#2}^{#3}}}
\begin{document} 
\title[Schr\"odinger operators and Hodge Laplacians]{The spectrum of 
Schr\"odinger operators and Hodge Laplacians on conformally 
cusp manifolds}
 
\author{Sylvain Gol\'enia} 
\address{Mathematisches Institut der Universit\"at Erlangen-N\"urnberg
Bismarckstr.\ 1 1/2 \\
91054 Erlangen, Germany}
\email{golenia@mi.uni-erlangen.de} 
\author{Sergiu Moroianu} 
\thanks{The authors were partially supported by the
contract MERG 006375, funded through the European Commission.
The second author was partially supported from the contracts
2-CEx06-11-18/2006 and
CNCSIS-GR202/19.09.2006.}
\address{Institutul de Matematic\u{a} al Academiei Rom\^{a}ne\\ 
P.O. Box 1-764\\RO-014700 
Bucha\-rest, Romania} 
\email{moroianu@alum.mit.edu} 
\date{\today} 
\begin{abstract} 
We describe the spectrum of the $k$-form Laplacian on conformally cusp
Riemannian manifolds. The essential spectrum is shown to vanish
precisely when the $k$ and $k-1$ de Rham cohomology groups of the
boundary vanish. We give Weyl-type asymptotics for the
eigenvalue-counting function in the purely discrete case. In the
other case we analyze the essential spectrum via positive commutator
methods and establish a limiting absorption principle. 
This implies the absence of the singular spectrum for a wide
class of metrics. We also exhibit 
a class of potentials $V$ such that the
Schr\"odinger operator has compact resolvent, although $V$ tends
to $-\infty$ in most of the infinity. We correct a statement
from the literature regarding the essential spectrum
of the Laplacian on forms on hyperbolic manifolds of finite volume, and
we propose a conjecture about the existence 
of such manifolds in dimension four whose cusps are rational homology spheres.

\end{abstract} 
\maketitle 
 
\section{Introduction} 
There exist complete, noncompact manifolds on which the
scalar Laplacian has purely discrete spectrum, see e.g.\
\cite{donelili}. The goal of this paper is to understand 
such phenomena for the Laplacian on differential forms in a more geometric 
setting. We aim to provide eigenvalue asymptotics whenever the 
spectrum is purely discrete, and to clarify the nature of the essential 
spectrum when it arises. 

We study $n$-dimensional non-compact Riemannian manifolds $X$ 
which outside a compact set are
diffeomorphic to a cylinder $(0,\infty)\times M$, where $M$ is a closed, 
possibly disconnected Riemannian manifold. The metric on $X$ near infinity 
is assumed  to be quasi-isometric to the unperturbed model metric 
\begin{align}\label{mc} 
g_p=y^{-2p}(dy^2+h),&&y\gg 0\end{align}
where $h$ is a metric on $M$ independent of $y$, and $p>0$. 
For $p=1$, this includes the case 
of finite-volume complete hyperbolic manifolds, which we discuss 
in some detail in the last section of the paper. 
The manifold $X$ is incomplete
if and only if $p>1$, and $\vol(X)<\infty$ if and only if $p>1/n$. For $p>1$, 
the metric \eqref{mc} is of metric-horn type as in \cite{lepe}.

We denote by $\Delta=d^*d+\delta^* \delta$ the Hodge Laplacian defined on 
smooth forms with compact support in $X$. It is a symmetric 
non-negative operator in $L^2(X, \Lambda^*X, g_p)$ and we also denote
by $\Delta$ its self-adjoint Friedrichs extension. If $p\leq 1$, i.e.\ 
if $(X,g_p)$ is complete, then $\Delta$ is essentially self-adjoint,
see \cite{Ga}.  Since $\Delta$ preserves the space of $k$-forms, we
can define $\Delta_{k}$  as its restriction to $\Lambda^kX$, which is
also self-adjoint.  

Let us advertise one application of our results. 
The essential spectrum of the Laplacian acting on forms on non-compact
manifolds has been extensively studied, as it provides
informations on the Hodge decomposition of the space of $L^2$
forms. Without attempting to give an exhaustive bibliography,
we mention here the papers \cite{bue, Carron, GW, lott, maz, mazphi}.
In Section \ref{s:betti}, we show:
\begin{proposition}\label{p:introhodge}
Let $(X,g)$ be a complete non-compact hyperbolic manifold of finite volume. 
Let $M$ be the boundary at infinity. 
Let $\tilde g$ be the metric $(1+\rho)g$, with 
$\rho\in\cC^\infty(X,\R)$, $\inf(\rho)>-1$ such that
\begin{eqnarray}
\rho(x)\rightarrow 0, \mbox{ as } x\rightarrow M \mbox{ and } 
\|d\rho\|_\infty<\infty
\end{eqnarray}
If $n=\dim(X)$ is odd, suppose also that the 
Betti number $b_{\frac{n-1}{2}}(M)$ vanishes. 
Then $\im(d)$ and $\im(\delta_{\tilde g})$ are closed, and  
\begin{equation*}
L^2(X, \Lambda^*X, \tilde g)= \ker(\Delta_{\tilde g})\oplus \im(d)\oplus
\im(\delta_{\tilde g}).	 
\end{equation*}
\end{proposition}
This proposition follows from the fact that
the essential spectrum $\siges(\Delta)$ of $\Delta$ does not contain $0$. 
For hyperbolic manifolds of finite volume, $\siges(\Delta)$
was computed by Mazzeo and Phillips \cite[Theorem 1.11]{mazphi},  
but their statement contains a gap, see Section \ref{s:betti} 
for geometric counterexamples. In \cite{francesca1}, 
Antoci also computes $\siges(\Delta_k)$ but for technical reasons 
she is unable to decide whether $0$ is isolated in the essential 
spectrum of $\Delta_k$ except in the case where $M=S^{n-1}$ with the
standard metric. For the metric \eqref{mc} and for a general
$M$, we compute $\siges(\Delta_k)$ in Proposition \ref{p:thema} and
we deduce that $0$ is never isolated in it for any $k$. 
In the Appendix, we also give a self-contained proof of the stability 
of the essential spectrum of the Laplacian  acting on forms for 
a large class of perturbations of the metric.

We investigate first in this paper
the absence of the essential spectrum and 
improve along the way the results of \cite{francesca1}.
We replace the condition $M=S^{n-1}$ with a 
weaker topological condition on the boundary at infinity $M$. Related results
were obtained in \cite{baer, GMo1, wlom}. The following theorem holds
for the conformally cusp metric  \eqref{cume2}, a generalization of
\eqref{mc}. To our knowledge, the result on eigenvalue asymptotics in this 
context is entirely new.

\begin{theorem}\label{t:Ith}
Let $X$ be a $n$-dimensional conformally cusp manifold 
(Definition \ref{dccm}) for some $p>0$. Fix an integer $k$ between 
$0$ and $n$. If the Betti numbers $b_k(M)$ and $b_{k-1}(M)$ 
of the boundary at infinity $M$ both vanish, then:
\begin{enumerate} 
\item the Laplacian $\Delta_k$ acting on $k$-forms on $X$ 
is essentially self-adjoint in $L^2$ for the metric $g_p$;
\item the spectrum of $\Delta_{k}$ is purely discrete;
\item the asymptotic of its eigenvalues, in the limit $\lambda\to\infty$,
is given by 
\begin{equation}\label{e:Ithmag}
N_{p}(\lambda) \approx \begin{cases}
C_1\lambda^{n/2}&
\text{for $1/n< p$,}\\
C_2\lambda^{n/2}\log \lambda &\text{for $p=1/n$,}\\
C_3\lambda^{1/2p}&\text{for $0<p<1/n$}.
\end{cases}
\end{equation}
\end{enumerate} 
\end{theorem}
Note that the hypothesis $b_k(M)=b_{k-1}(M)=0$ does not hold
for $k=0,1$; it also does not hold for $k=n, n-1$ if $M$ has at least 
one orientable connected component. In particular, Theorem \ref{t:Ith} does
not apply to the Laplacian acting on functions. 
We refer to Section \ref{s:betti} for a discussion of this hypothesis,
implications about hyperbolic manifolds and open problems in dimension 
$4$ and higher.

The constants $C_1$, $C_2$ are given by \eqref{e:Ic1}, \eqref{e:Ic2}. Up to a
universal constant which only depends on $\dim(X)$, they are just
the volume of $(X,g_p)$ in the finite volume case $p>1/n$ (here we get 
the precise form of the Weyl law for closed manifolds), respectively 
the volume of the boundary at infinity $M$ with respect to a naturally 
induced metric in the case $p=1/n$.
When the metric $g_0$ is exact (see Definition \ref{defex}; this includes 
the metric \eqref{mc}), $C_3$ is 
given by \eqref{C_3}. 

We stress that $\Delta_k$ is essentially self-adjoint and has purely 
discrete spectrum solely based on the  hypothesis 
$b_k(M)=b_{k-1}(M)=0$, without any condition like completeness 
of the metric or finiteness of the volume, see Corollary
\ref{cor3.2}. Intuitively, the continuous  spectrum of $\Delta_k$ is 
governed by zero-modes of the form Laplacian  on $M$ in dimensions $k$
and $k-1$ (both dimensions are involved  because of algebraic
relations in the exterior algebra). By Hodge theory, the kernel of
the $k$-form Laplacian on the compact manifold  $M$ is isomorphic to
$H^k(M)$, hence these zero-modes (harmonic forms)  
exist precisely when the Betti numbers do not vanish. 
We remark that, in the study of the scalar magnetic Laplacian \cite{GMo1} 
and of the Dirac operator \cite{baer}, the r\^ole of the Betti numbers 
was played by an integrality condition on the magnetic field, 
respectively by a topological condition on the spin structure on the cusps.

We are now interested in some refined -- and less studied -- properties of the
essential spectrum, namely the absence of singularly continuous 
spectrum and weighted estimates of the resolvent, 
which gives non-trivial dynamical properties on the group $e^{it\Delta_k}$. 
For the metric
\eqref{mc}, in the complete case, a refined analysis was started in
\cite{francesca2}. We determine the nature
of the essential spectrum by positive commutator techniques. 
The case of the Laplacian on functions has been treated originally by
this method in \cite{FH}, and by many other methods in the literature,
see for instance \cite{Guillope, Kumura} for different
approaches. In \cite{GMo1} we introduced a conjugate operator 
which was ``local in energy'', in order to deal with a bigger class of
perturbations of the metric. We use the same idea here, however 
the analysis of the Laplacian on $k$-forms turns out to be more involved
than that of the scalar magnetic Laplacian. Indeed,
one could have two thresholds and the positivity is 
harder to extract between them. The difficulty arises from the compact part
of the manifold, since we can diagonalize the operator only
on the cusp ends. The resolvent of the operator does \emph{not} stabilize this
decomposition. To deal with this, we introduce a perturbation of the
Laplacian which uncouples the compact part from the cusps in a gentle way. 

Let $L$ be the operator on $\cunc(X,\Lambda^* X)$ of multiplication by
a fixed smooth function $L:X\to [1,\infty)$ defined by:
\begin{align}\label{e:L}
 L(y):=\begin{cases} 
\ln(y)& \text{for $p=1$},\\ 
\frac{y^{1-p}}{1-p}& \text{for $p< 1$},
\end{cases} && \mbox{ for } y\geq y_0 \mbox{ and for some } y_0\geq 1.
\end{align}  
 Given $s\geq 0$, let
$\Lr_{s}$ be the domain of $L^s$ equipped with the graph norm. We set
$\Lr_{-s}:=\Lr_{s}^*$ where the adjoint space is defined so that
$\Lr_{s}\subset  L^2(X,\Lambda^* X, g_p) \subset \Lr_{s}^*$, using the
Riesz lemma.  Given $I\subset \R$, let $I_{\pm}$ be the set of complex
numbers $a\pm ib$, where $a\in I$ and $b>0$.  

Perturbations of \emph{short-range} (resp.\ \emph{long-range}) type
are denoted  with the subscript ${\rm sr}$ (resp.\ ${\rm lr}$); they
are supported in $(2, \infty)\times M$. We ask long-range
perturbations to be \emph{radial}. In other words, a perturbation
$W_{\rm lr}$ satisfies $W_{\rm lr}(y,m)= W_{\rm lr}(y,m')$ for all
$m,m'\in M$. 

\begin{theorem}\label{t:Ith1}
Let $\varepsilon>0$. We consider the metric $\tilde g=(1+\rho_{\rm sr} +
\rho_{\rm lr})g_p$, with $0<p\leq 1$ and where the short-range and
long-range components satisfy 
\begin{equation}\label{e:Iconf}\begin{split}
L^{1+\varepsilon}\rho_{\rm sr },
d\rho_{\rm sr } \text{ and } \Delta_g\rho_{\rm sr } \in L^\infty(X),\\
L^{\varepsilon}\rho_{\rm lr }\text{, } 
 L^{1+\varepsilon}d\rho_{\rm lr } \text{ and } \Delta_g \rho_{\rm lr } 
\in L^\infty(X).
\end{split}\end{equation}
Suppose that at least one of the two Betti numbers $b_k(M)$ and $b_{k-1}(M)$ 
is non zero . Let $V= V_{\rm loc}+ V_{\rm sr}$ and $V_{\rm
lr}$ be some potentials, where $V_{\rm loc}$ is measurable with
compact support and $\Delta_{k}$-compact, and $V_{\rm sr}$ and
$V_{\rm lr}$  are in $L^\infty(X)$ such that:      
\begin{equation*}
\|L^{1+\varepsilon}V_{\rm sr}\|_\infty<\infty,\, V_{\rm lr}\rightarrow
0, \mbox{ as } y\rightarrow +\infty  \mbox{  and }
\|L^{1+\varepsilon}d V_{\rm  lr}\|_\infty<\infty. 
\end{equation*} 
Consider the Schr\"odinger operators  $H_0=
\Delta_{k, p}+V_{\rm lr}$ and $H=H_0 + V$. Then 
\begin{enumerate}
\item The essential spectrum of $H$ is $[\inf\{\kappa(p)\},\infty)$,
where the set of thresholds $\kappa(p)\subset \rz$ is
defined in \eqref{e:kappa}. \label{a)}
\item $H$ has no singular continuous spectrum. \label{b)}
\item The eigenvalues of $H$ have finite multiplicity and no
accumulation points outside $\kappa(p)$.
 \label{c)}
\item \label{d)} Let $\cJ$ be a compact interval such that $\cJ\cap
\big(\kappa(p)\cup  \sigma_{\rm pp}(H)\big)=\emptyset$. Then, for
all $s$ in $(\frac12, \frac32)$, there exists $c$ such that  
\begin{equation*}
\|(H-z_1)^{-1} - (H-z_2)^{-1} \|_{\Bc(\Lr_s, \Lr_{-s})} \leq c
|z_1-z_2|^{s-1/2}, 
\end{equation*}
for all $z_1, z_2 \in \cJ_{\pm}$.
\item \label{e)} Let $\cJ=\R\setminus \kappa(p)$ and let $E_0$ and $E$ be the
continuous spectral component of $H_0$ and $H$, respectively. Then,
the wave operators defined as the strong limit
\begin{equation*}
\Omega_{\pm}=\mathrm{s-}\!\!\!\!\lim_{t\rightarrow \pm \infty} e^{itH}e^{-itH_0}E_0(\cJ) 
\end{equation*} 
exist and are complete, i.e.\ $\Omega_{\pm} \Hr= E(\cJ)\Hr$, where
$\Hr=L^2(X, \Lambda^k X, g_p)$.
\end{enumerate} 
\end{theorem}
The statement \eqref{a)} remains true for a wide family
of metrics asymptotic to  \eqref{mc}, see Proposition
\ref{p:thema}. The fact that every eigenspace  
is finite-dimensional (in particular, for the eigenvalue $\kappa(p)$, 
the bottom of the continuous spectrum) is due to the general 
\cite[Lemma B.1]{GMo1} and holds for an arbitrary {conformally} cusp metric
\eqref{cume2}.  
The proof of the rest, \eqref{b)}--\eqref{e)} where in \eqref{c)} 
we consider eigenvalues with energy 
different from $\kappa(p)$, relies on the Mourre theory \cite{ABG, mou} 
with an improvement for the regularity of the boundary value of the resolvent, 
see \cite{ggm} and references therein. These wide classes of
perturbation of the metric have been introduced in
\cite{GMo1}, however here the treatment of long-range
perturbations is different from the one in \cite{GMo1}. 
We prove these facts in  Section \ref{s:SRLR}. 
Note that this general class of metrics seems to be difficult 
to analyze by standard scattering techniques.

We now turn to question of the perturbation of the Laplacian $\Delta_k$ by 
some non relatively compact potential. In the Euclidean $\R^n$, using
Persson's formula,  one sees that $\sigma_{\rm ess}(H)$ is empty for
$H=\Delta+V$ if $V\in L^\infty_{\rm loc}$ tends to $\infty$ at
infinity. However, the converse is wrong as noted in \cite{S83} by
taking $V(x_1, x_2)=x_1^2x_2^2$ which gives rise to a compact 
resolvent. Morally speaking, a particle can not escape in the
direction of finite energy at infinity which is too narrow compared to
the very attracting part of $V$ which tends to infinity.  In our
setting, the space being smaller at infinity, it is easier to create
this type of situation even if most of the potential tends to
$-\infty$. To our knowledge, the phenomenon is new. The general statement 
appears in Proposition \ref{thschr}.

\begin{proposition}\label{p:intro}
Let $p>0$ and $(X,g_p)$ be a conformally cusp manifold.
Let $V$ in $ y^{2p}\cun(\oX)$ be a smooth potential with
Taylor expansion 
$y^{-2p}V=V_0+y^{-1}V_1+O(y^{-2})$ at infinity. Assume that
$V_0$ is non-negative and not identically zero on any connected 
component of $M$. Then the Schr\"odinger  operator $\Delta_{k}+V$
is essentially self-adjoint and 
has purely discrete spectrum. The eigenvalues  obey the generalized
Weyl law \eqref{e:Ithmag} and the constants $C_1,C_2$ do not depend on $V$. 
\end{proposition}

As in Theorem \ref{t:Ith}, the completeness of the manifold is not
required to obtain the essential self-adjointness of the operator.
Note that by assuming $p>1/2$ 
(in particular, on finite-volume hyperbolic manifolds, for which $p=1$) 
and $V_1<0$, we get
$V\sim y^{2p-1} V_1$ tends to $-\infty$ as we approach $M\setminus
\supp(V_0)$. The support of the non-negative leading term $V_0$ must be 
nonempty, but otherwise it can be chosen arbitrarily small.

Some of the results about the essential spectrum appeared in the
unpublished paper \cite{GMo}.

\subsection*{Notations.}Given two Hilbert spaces $\Hr$ and $\Kr$, we
denote by $\Bc(\Hr, \Kr)$ the space of bounded
operators from $\Hr$ to $\Kr$, and by
$\Kc(\Hr, \Kr)$ the subspace of compact operators. When $\Kr=\Hr$, we
simply write $\Bc(\Hr)$ and $\Kc(\Hr)$.
Given a closed operator $H$ acting in $\Hr$, we denote by
$\Dc(H)$ its domain and we endow it with the graph norm $\|\cdot\| +
\|H\cdot\|$. We denote by $\sigma(H)$ its spectrum and by $\rho(H)$
its resolvent set. Given a vector bundle $E$ over a smooth manifold
$X$, we denote by $\cC_c^\infty(X, E)$ the space of smooth sections in
$E$ with  compact support.

\subsection*{Acknowledgements.}
We acknowledge useful discussions with Vladimir Geor\-ges\-cu
and  Andreas Knauf.
 
\section{Geometric definitions}\label{lcm}

Let $\oX$ be a smooth $n$-dimensional compact manifold
with closed boundary $M$, and $x:\oX\to[0,\infty)$ a
boundary-defining function. Let $\cI\subset\cun(\oX)$ be the principal 
ideal generated by the function
$x$. A \emph{cusp vector field} is a smooth vector field $V$ on
$\oX$ such that $dx(V)\in\cI^2$. The space $\ccV$ of cusp vector fields forms a
Lie subalgebra $\ccV$ of the Lie algebra of smooth vector fields on $\oX$.
Moreover, there exists a smooth vector bundle $\ctX\to \oX$ whose space of 
section identifies naturally with $\ccV$.

A \emph{cusp metric} on $\oX$ is a (smooth) Euclidean metric $g_0$ on the 
bundle $\ctX$ over $\oX$. Since $\ctX$ and $TX$ are canonically identified over
the interior $X:=\oX\setminus M$, $g_0$ induces a
complete Riemannian metric  on $X$. 
We want to study the Laplacian on $k$-forms associated to the
metric $g_p:=x^{2p}g_0$ for fixed $p>0$.

Fix a product decomposition of $\oX$ near $M$, i.e., an embedding 
$[0,\varepsilon)\times M\hookrightarrow \oX$ compatible with the function $x$.
Then near the boundary, $g_p$ and the cusp metric $g_0$ take the form 
\begin{align}\label{cume2}
g_0=a \left(\frac{dx}{x^2}+\alpha(x)\right)^2+h(x), &&g_p:=x^{2p}g_0,&& p>0
\end{align}
where $a\in\cun(\oX)$ such that $a_{|M}>0$,  $h$ is a smooth
family of symmetric $2$-tensors on $M$ and $\alpha$ is a
smooth family of $1$-forms in $\cun([0,\varepsilon)\times
M,\Lambda^1(M))$.  Note that the
metric \eqref{mc} is a particular case of the metric $g_p$ from
\eqref{cume2} with $\alpha\equiv 0, a\equiv 1$
and $h(x)$ constant in $x$ (define $x:=1/y$). 

\begin{definition}\label{dccm}
A Riemannian manifold $(X,g_p)$ which is the interior of a compact 
manifold with boundary together with a Riemannian metric $g_p$ as in 
\eqref{cume2} for some $p>0$ is called a \emph{conformally cusp manifold}.
The boundary $M=\partial X$ may be disconnected.
\end{definition}

By \cite[Lemma 6]{wlom}, the function $a_0:=a(0)$, 
the metric $h_0:=h(0)$ and the class modulo exact forms 
of the $1$-form $\alpha_0:=\alpha(0)$, defined on $M$,
are independent of the chosen product decomposition and of 
the boundary-defining function $x$ inside the fixed cusp structure.
We also recall the following definition.

\begin{definition}\label{defex}
The metric $g_0$ is called \emph{exact} if $a_0=1$ and $\alpha_0$ is an
exact $1$-form.
\end{definition}

Let $E,F\to\oX$ be smooth vector bundles. The space of cusp differential
operators $\Diffc(\oX,E,F)$ is the space of those differential operators
which in local trivializations can be written as composition of cusp
vector fields and smooth bundle morphisms. The \emph{normal operator}
of $P\in\Diffc(\oX,E,F)$ is defined by  
\[\rz\ni\xi\mapsto\cN(P)(\xi):= 
\left(e^{i\xi/x}Pe^{-i\xi/x}\right)_{|x=0}\in\Diff(M,E_{|M},F_{|M}).\] 
From the definition, $\ker \cN=\cI\cdot\Diffc$, which we denote again
by $\cI$.
\begin{example}\label{ex1}Given a family $(P_x)_{x\in [0,\epsilon)}$
of operators on $M$ depending smoothly on $x$, one has
$\cN(P_x)(\xi)=P_0$. Also, $\cN(x^2\px)(\xi)=i\xi$.
\end{example}
From the definition, the normal operator map is linear and multiplicative.
Let $P\in\Diffc(\oX,E,F)$ be a cusp operator and $P^*$ its adjoint with
respect to $g_0$. Then $\cN(P^*)(\xi)$ is the adjoint of $\cN(P)(\xi)$
for the volume form ${a_0}^{1/2} dh_0$
and with respect to the metric on $E_{|M},F_{|M}$ induced by restriction. 
Indeed, since $\cN$ commutes 
with products and sums, it is enough to check the claim for the set of 
local generators of $\Diffc$ from example \ref{ex1}, which is a 
straightforward computation.

\section{Proof of Theorem \ref{t:Ith}}

We follow the ideas of \cite{wlom} and \cite{GMo1}. We will first show that
$\Delta_k$ is $x^{-2p}$ times an elliptic cusp differential operator. Since 
we work on bundles, we first trivialize the bundles of forms in the $x$ 
direction. Near $M$, set 
\begin{align}\label{vi} 
V_0:=x^2\px\in\ccV,&& V^0:=x^{-2}dx+\alpha.
\end{align}
We get an orthogonal decomposition of smooth vector bundles 
\begin{equation}\label{decom} 
\Lambda^k(\ctX)\simeq\Lambda^k(TM)\oplus V^0\wedge\Lambda^{k-1}(TM),
\end{equation} 
where $\Lambda^*(TM)$ is identified with the kernel of the contraction
by $V_0$. 
 
The de Rham differential
$d:\cun(X,\Lambda^k X)\to \cun(X,\Lambda^{k+1}X)$
restricts to a cusp differential operator
$d:\cun(\oX,\Lambda^k(\ctX))\to \cun(\oX,\Lambda^{k+1}(\ctX))$. Its
normal operator in the decomposition \eqref{decom} is 
\begin{equation}\label{e:Nd}
\cN(d)(\xi)= \begin{bmatrix} 
d^M-i\xi\alpha_0\wedge&d^M\alpha_0\wedge\\ 
i\xi&-(d^M-i\xi\alpha_0\wedge) 
\end{bmatrix}.
\end{equation} 
 
The principal symbol of a cusp operator in
$\Diff_c^*(\oX,E,F)$  
extends as a map on the cusp cotangent bundle. An operator $P\in
x^{-2p}\Diff_c^*(\oX,E,F)$ is called \emph{cusp-elliptic}  in the
sense of Melrose if the principal symbol of $x^{2p} P$ is invertible
on $\ctsX\setminus\{0\}$ down to 
$x=0$; it is called \emph{fully elliptic}  if it is cusp-elliptic and
if the differential operator 
$\cN(x^{2p} P)(\xi)$ on $M$ is invertible as an unbounded
operator in $L^2(M,E_{|M},F_{|M})$ for all $\xi\in\rz$, see \cite{meleta}. 
 
\begin{proposition}\label{th4} 
The Laplacian $\Delta_{k}$ of the metric $g_p$ belongs to
$x^{-2p}\Diffc^2(\oX,\Lambda^k(\ctX))$ and is cusp-elliptic. 
Moreover, $x^{2p}\Delta_{k}$
is fully elliptic if and only if the de Rham coho\-mo\-lo\-gy groups
$H^k_{\mathrm{dR}}(M)$ and $H^{k-1}_{\mathrm{dR}}(M)$ both vanish.
\end{proposition} 
\begin{proof} 
The principal symbol of the Laplacian of $g_p$ on $\Lambda^k X$ is
$g_p$ times the identity. Since  $x^{-2p}g_p=g_0$ on the cotangent
bundle, and since $g_0$ extends by definition to a positive-definite
bilinear form on $\ctsX$, it follows that $x^{2p}\Delta_{k}$ is cusp-elliptic. 
Let $\delta_0^k, \delta_p^k$ be the formal adjoint of $d:\Lambda^k X 
\to\Lambda^{k+1}(X)$ with respect to $g_0$,
resp.\ $g_p$. Then $\delta_p^k=x^{(2k-n)p+2}\delta_0 x^{(n-2(k+1))p-2}$. 
By conjugation invariance of the normal operator, we obtain
$\cN\big(x^{2p}(d\delta_p+\delta_p
d)\big)=\cN(d\delta_0+\delta_0d)$. By Hodge theory, 
the kernel of $\cN(d\delta_0+\delta_0d)(\xi)$  is
isomorphic to the cohomology of the complex $\big(\Lambda^*(\ctX)_{|M},
\cN(d)(\xi)\big)$. We write 
\begin{align*} 
\cN(d)(\xi)=A(\xi)+B(\xi), \mbox{ where }
A(\xi)=\begin{bmatrix}0&0\\i\xi&0\end{bmatrix},
B(\xi)=\begin{bmatrix} 
d^M-i\xi\alpha_0\wedge&d^M\alpha_0\wedge\\ 
0&-(d^M-i\xi\alpha_0\wedge) 
\end{bmatrix}.
\end{align*} 
We claim that for $\xi\neq 0$ the cohomology of $\cN(d)(\xi)$ vanishes.
The idea is to use again Hodge theory but with respect to the volume
form $dh_0$ on $M$. Then $B(\xi)^*$ anti-commutes with $A(\xi)$ and similarly
$A(\xi)^*B(\xi)+B(\xi)A(\xi)^*=0$. Therefore
\[\cN(d)(\xi)\cN(d)(\xi)^*+ 
\cN(d)(\xi)^*\cN(d)(\xi)=\xi^2 I+B(\xi)^*B(\xi)+B(\xi)B(\xi)^*\] 
where $I$ is the $2\times 2$ identity matrix. So for $\xi\neq 0$
the Laplacian of $\cN(d)(\xi)$ is a strictly positive elliptic operator, hence 
it is invertible. 
 
Let us turn to the case $\xi=0$. We claim that the cohomology 
of $(\Lambda^*M\oplus\Lambda^{*-1}M, \cN(d)(0))$ is isomorphic to 
$H^*_{\mathrm{dR}}(M)\oplus H^{*-1}_{\mathrm{dR}}(M)$. Indeed, 
notice that 
\[\cN(d)(0)=\begin{bmatrix}1&-\alpha_0\wedge\\0&1\end{bmatrix} 
\begin{bmatrix}d^M&0\\0&-d^M\end{bmatrix} 
\begin{bmatrix}1&\alpha_0\wedge\\0&1\end{bmatrix}.\] 
In other words, the differential $\cN(d)(0)$ is conjugated to 
the diagonal de Rham differential, so they have isomorphic cohomology. 
\end{proof} 

\begin{corollary}\label{cor3.2}
If $b_k(M)=b_{k-1}(M)=0$ then for every $p>0$, the  Laplacian $\Delta_{k}$ 
of the metric $g_p$ is essentially self-adjoint on
$\cC^\infty_c(X, \Lambda^k X)$ and has purely discrete spectrum. 
The domain of its self-adjoint extension is the weighted cusp Sobolev
space $x^{2p}H^2_c(X,\Lambda^kX)$. 
\end{corollary}

Note that for $p\leq 1$ the metric $g_p$ is complete, in which case
$\Delta_k$ is essentially self-adjoint without any extra hypothesis, 
see \cite{Ga}; however, even in the complete case, one cannot 
describe the domain of the unique self-adjoint extension
if the Betti numbers do not vanish. 
The definition of the cusp Sobolev spaces is recalled in the proof below.
\begin{proof}
By the de Rham theorem, the vanishing of the Betti numbers is equivalent to
the vanishing of the de Rham cohomology groups $H^k_{\mathrm{dR}}(M)$ and 
$H^{k-1}_{\mathrm{dR}}(M)$. Hence from Proposition \ref{th4} it follows that
$x^{2p} \Delta_{k}$ is fully elliptic. From the  general properties of
the cusp calculus \cite{meni96c}, there exists a Green
operator $G\in x^{2p}\Psi_c^{-2}(X,\Lambda^k X)$  which inverts
$\Delta_k$ up to remainders in the ideal
$x^\infty\Psi_c^{-\infty}(X,\Lambda^k X)$. Although we do not use
  it here, recall that $\Psi_c^k(X,E,F)$ is defined as a space of  
distributional kernels on $X\times X$ 
(with coefficients in the bundle $E\boxtimes F^*$) which are classically 
conormal to the diagonal and with prescribed asymptotics near the boundary.
More precisely, they must lift to be extendible across the 
front face of a double blow-up resolution
of the corner $M\times M$ in $\oX\times \oX$, and to vanish in Taylor series
at the other boundary hyperfaces. We refer e.g.\ to \cite{wlom}
for the precise definition. Once the properties of the calculus
have been established, the existence of the Green operator $G$ is 
a standard application of elliptic theory.

The Sobolev space $H^q_c(X,\Lambda^k X)$ is by definition the
intersection of  the domains of the maximal extensions of all elliptic
cusp operators  inside $\Psi_c^q(X,\Lambda^k X)$. Cups operators
of order $r\in\cz$ map  $H^q_c(X)$ to $H^{q-\Re(r)}_c(X)$, see
\cite{wlom}. 

Look now at $\Delta_k$ restricted to $\cunc(X,\Lambda^k X)$. It
is easy to see that  $x^{2p}H^2_c(X)$ is contained in the domain of
the  minimal extension of $\Delta_k$. Conversely, using the Green
operator $G$  and the mapping properties of cusp operators stated
above,  we see that every vector in the domain of the maximal extension of
$\Delta_k$ belongs to $x^{2p}H^2_c(X)$. In conclusion, the minimal and the 
maximal domain are the same and equal $x^{2p}H^2_c(X)$.

Recall now from \cite{wlom} that for $p,q>0$, operators in 
$x^{p}\Psi_c^{-q}(X)$ are compact. Since the
self-adjoint operator $\Delta_k$ has a compact  
inverse modulo compact operators, it follows that it has purely discrete 
spectrum.
\end{proof}

Notice that we have more generally proved that a symmetric 
fully elliptic cusp operator of order $(-p,-r)$ with $p,r>0$ 
on a conformally cusp manifold is essentially self-adjoint and has purely
discrete spectrum.
 
We can now conclude the proof of Theorem \ref{t:Ith}.
By Proposition \ref{th4}, if we assume that $b_k(M)=b_{k-1}(M)=0$ 
it follows that $\Delta_{k}$ is fully elliptic. From Corollary \ref{cor3.2},
this implies (by a form of elliptic regularity) that $\Delta_k$ is 
essentially self-adjoint and the domain of the extension is the weighted
Sobolev space $x^{2p}H^2(X)$. By \cite[Theorem 17]{wlom}, the
spectrum of $\Delta_{k}$ is purely discrete and accumulates towards
infinity according to \eqref{e:Ithmag}, modulo identification of the correct
coefficients. This is proved in two steps as in \cite{wlom}: first, 
the complex powers of a self-adjoint fully elliptic cusp operator belong 
again to the cusp calculus and form an analytic family; secondly, the trace 
of an analytic family in the complex variable $z$ 
of cusp operators of order $(-z,-pz)$ is well-defined for $z<-n, pz<-1$,
and extends to a meromorphic function on $\cz$ with at most double poles at 
certain reals. By \cite[Proposition 14]{wlom} and Delange's theorem 
(\cite[Lemma 16]{wlom}), the coefficients $C_1,C_2,C_3$ are determined by 
the order of the first occurring pole of the zeta function (i.e., the 
smallest $z\in\rz$ which is a pole for
$\zeta(z):=\Tr(\Delta_k^{-z/2})$ and by its 
leading coefficient in Laurent expansion. The principal
symbol $\sigma_1(\Delta_{k}^{1/2})$ is identically $1$ on the cosphere
bundle. The dimension of the form bundle equals the binomial coefficient 
$\binom{n}{k}$. 
\subsection{The case $p>1/n$.} From \cite[Proposition 14]{wlom}, the 
first pole of $\zeta$ is simple, located at $z=n$ with residue 
\[R_1=(2\pi)^{-n} \dbinom{n}{k} \vol(X)\vol(S^{n-1}).\]
From \cite[Lemma 16]{wlom}, we get the asymptotic equivalence for the 
eigenvalues of $\Delta_{k}^{1/2}$: 
\[N(\Delta_{k}^{1/2},\lambda)\approx R_1/n \lambda^n.\]
Taking into account $N(\Delta_{k}^{1/2},\lambda^{1/2})=N(\Delta_{k},\lambda)$, 
we get 
\begin{equation}\label{e:Ic1}
C_1={\dbinom{n}{k}}\frac{\Vol(X,g_p)\Vol(S^{n-1})}{n(2\pi)^n}.
\end{equation} 
\subsection{The case $p=1/n$.}
From \cite[Proposition 14]{wlom}, the 
first pole of $\zeta$ is double, located at $z=n$ with leading coefficient
\[R_2=n(2\pi)^{-n} \dbinom{n}{k} \vol(M)\vol(S^{n-1}).\]
From \cite[Lemma 16]{wlom}, the eigenvalues of $\Delta_{k}^{1/2}$ obey
\[N(\Delta_{k}^{1/2},\lambda)\approx R_2/n \lambda^n \log \lambda.\]
Again translating from the counting function of $\Delta_{k}^{1/2}$
to that of $\Delta_{k}$, we get
\begin{equation}\label{e:Ic2}
C_2={\dbinom{n}{k}} \frac{\Vol(M,h_0)\Vol(S^{n-1})}{2(2\pi)^n}.
\end{equation} 

\subsection{The case $p<1/n$.} In this situation the 
first pole of $\zeta$ is simple, located at $z=1/p$ with residue
\[-\frac{n}{2\pi}\int_\rz 
\Tr\  \cN\left(x^{-1}\Delta_k^{-\frac{1}{2p}}\right)(\xi)d\xi.\]
To compute $C_3$, we suppose also that the metric $g_0$ is exact. 
With this assumption, by replacing $x$ with another boundary-defining
function inside the same cusp structure, we can assume that
$\alpha_0=0$ (see \cite{wlom}). In this case, \eqref{e:Nd} gives
\begin{equation}\label{cndkp}
\cN(x^{2p}\Delta_{k})(\xi)=\begin{bmatrix}\xi^2+\Delta_k^M&0\\0& 
\xi^2+\Delta_{k-1}^M 
\end{bmatrix}.
\end{equation}
This allows us to compute the integral from \cite[Proposition 14]{wlom} 
in terms of the zeta functions of the Laplacians on forms on $M$ with 
respect to $h_0$. Straightforwardly, one gets
\begin{equation}\label{C_3}
C_3=\frac{\Gamma\left(\tfrac{1-p}{2p}\right) 
\left(\zeta\left(\Delta_k^M, \frac{1}{p}-1\right) 
+\zeta\left(\Delta_{k-1}^M, \frac{1}{p}-1\right)\right)} 
{2\sqrt{\pi}\Gamma\left(\tfrac{1}{2p}\right)}.
\end{equation} 
This end the proof of Theorem \ref{t:Ith}.

\section{Schr\"odinger operators and discrete spectrum}\label{schro}

In this section, we prove the compactness of the resolvent of the 
Schr\"odinger operator for a
class of potentials that tend to $+\infty$ only towards a very small
part of the infinity. Proposition \ref{p:intro} for the metric 
\eqref{cume2} is a particular case of this analysis.

Let $H_0$ be a cusp pseudodifferential operator on $X$. We say that $H_0$
has the \emph{unique continuation property at infinity} if for all
$\xi\in\rz$, the normal operator $\cN(H_0)(\xi)$ has the (weak) unique
continuation property as  an operator on each connected component of
$M$, i.e., the non-zero  solutions $\phi$ to the pseudodifferential
equation $\cN(H_0)(\xi)\phi=0$ do not vanish on any open set.
 
\begin{proposition}\label{thschr}
Let $g_p$ be the metric on $X$ given by \eqref{cume2} near
$\partial X$. Let $H_0$ be a non-negative cusp-elliptic 
operator, $H_0\in x^{-qp}\Psi_c^q(X,E)$ for some $q>0$. Assume that
$x^{qp}H_0$ has the unique continuation property at infinity.
Let $V$ be a self-adjoint potential in $x^{-qp}\cun(\oX,E)$. 
Assume  $V_0:=(x^{qp}V)_{|M}\in\cun(M,E_{|M})$
is semi-positive definite and in each connected 
component of $M$ there is $z$ with $V_0(z)>0$. 
Then $H:=H_0+V$ is essentially self-adjoint in $L^2(X,E)$ and
$\sigma_{\rm ess}(H)=\emptyset$.  Its eigenvalue counting function, as
$\lambda$ goes to infinity, satisfies 
\begin{equation*} 
N_{H}(\lambda) \approx \begin{cases} 
C_1'\lambda^{n/q}& \text{for $1/n< p<\infty$,}\\
C_2'\lambda^{n/q}\log \lambda &\text{for $p=1/n$} \\
C_3'\lambda^{\frac{1}{qp}} &\text{for $p<1/n$}.
\end{cases} 
\end{equation*} 
\end{proposition} 

\begin{proof} 
We start by proving that the operator $H$ is fully
elliptic. Let $\xi\in\R$. A non-zero solution $\phi$ of
$\cN(x^{qp}H)(\xi)\phi=0$ satisfies $\langle\cN(x^{qp}H_0) (\xi)
\phi,\phi\rangle+\langle V_0\phi,\phi\rangle=0$.
The operators $\cN(x^{qp}H_0)(\xi)$ and $V_0$ being non-negative, we
get $\cN(x^{qp}H_0)(\xi)\phi=0$ and $V_0\phi=0$. 
By unique continuation, solutions of the elliptic operator 
$\cN(x^{qp}H_0)(\xi)$ which are not identically zero 
on a given connected component of $M$ do not vanish on any open subset 
of that component. However since $V_0\phi=0$ and $V_0(z)>0$, $\phi$ must vanish
in the neighborhood of $z$ where $V_0$ is invertible. 
Contradiction. $H$ is fully elliptic.

By \cite[Lemma 10 and Corollary 13]{wlom}, it follows that $H$
is essentially self-adjoint with domain $x^{qp}H^{q}_c(M)$,
and has purely discrete spectrum.
By \cite[Proposition 14]{wlom}, the constants $C_1'$ and $C_2'$ can be
computed as in Section \ref{lcm}, they depend only 
on the principal symbol of $H$ and so they are independent of $V$. The
coefficient $C_3'$ depends only on $\cN(H)$.\end{proof}  

The unique continuation property holds for instance when $\cN(H_0)(\xi)$ is
an elliptic second-order differential operator for all $\xi$, in
particular for the Laplacians $\Delta_{k}$ on differential forms or the
(scalar) magnetic Laplacian, like in \cite{GMo1}. Thus Proposition \ref{thschr} 
applies to $\Delta_{k}+V$ for any cusp metric. 
The constants $C_1',C_2'$ are 
still  given by  \eqref{e:Ic1}, \eqref{e:Ic2} since 
in \cite[Proposition 14]{wlom} only the principal symbol plays a r\^ole
for $p\geq 1/n$.
The coefficient $C_3'$ can
be computed if we assume that the metric is exact.  It will depend on
the zeta function of $\Delta^M+V_0$. The computation is similar to
\eqref{C_3}.  

Concerning essential self-adjointness, the hypothesis on the
regularity of the potential part can be weakened using \cite{BMS} for
elliptic operators of order $2$. In the result on the absence
of the essential spectrum, one can replace $V$ by $W\in
L^\infty_{\rm loc}$ where $V-W$ tends to $0$ as $x$ tends to infinity
using the Rellich-Kondrakov lemma and the ellipticity of $H_0$.

\section{The analysis of the essential spectrum}
In this Section we prove Theorem \ref{t:Ith1} part \eqref{a)}, see
Proposition \ref{p:thema}. We also diagonalize the Laplacian in two different 
ways. The first one, carried out in Section 
\ref{nonem}, exhibits some key invariant subspaces. It also allows us
to compute the essential spectrum. The second 
one, given in Section \ref{s:diago}, goes one step beyond and reformulates 
the problem in some ``Euclidean'' variables. This will be fully used for 
the positive commutator techniques, see Section \ref{s:Mourre}.

We fix $k\in\{0,\ldots,n\}$ and $p>0$ and introduce the constants
\begin{align}\label{c0c1}
c_0:=\big((2k+2-n)p-1\big)/2, && c_1:=\big((2k-2-n)p+1\big)/2.
\end{align}
The set of \emph{thresholds} is defined as follows:
\begin{align}\nonumber
&\text{for $p<1$, }\kappa(p)=\begin{cases}
\emptyset, & \text{if $b_k(M)=b_{k-1}(M)=0$}, \\
\{0\}, & \text{otherwise}.\\
\end{cases}\\
\label{e:kappa}
&\text{for $p=1$, }\kappa(p)=\big\{c_i^2\in\{c_0^2,c_1^2\}; 
b_{k-i}(M)\neq 0\big\}.\\
\nonumber
&\text{for $p>1$, }\kappa(p)=\emptyset.
\end{align}

\subsection{The high and low energy forms decomposition}\label{nonem} 
We proceed like in \cite{GMo1} and we restrict to metrics which near 
$M=\{x=0\}$ are of the form  
\begin{equation}\label{gp'}
g_p=x^{2p}\left(\frac{dx^2}{x^4}+ h\right).
\end{equation} 
Here $h$ is independent of $x$.
We fix $k$ and localize our computation to the end
$X':=(0,\varepsilon)\times M\subset X$. The objects we study do not
depend on $\varepsilon$. We introduce
\begin{eqnarray}\label{e:Kr}
 \Kr:=L^2\left((0,\varepsilon), x^{(n-2k)p-2} dx\right).
\end{eqnarray} 
 Using  \eqref{decom}, we get:   
\begin{equation*} 
L^2(X',\Lambda^k X)=\Kr\otimes 
\left(L^2(M,\Lambda^k M)\oplus
\frac{dx}{x^2}\wedge L^2(M,\Lambda^{k-1} M)\right). 
\end{equation*}  
Setting $\Hr_{{\rm l}_0}:=\Kr\otimes \ker(\Delta^M_k)$ and
$\Hr_{{\rm l}_1}:=\Kr\otimes \ker(\Delta^M_{k-1})$ and with a slight
abuse of notation, this gives 
\begin{equation}\label{dp} 
L^2(X',\Lambda^k X)=\Hr_{{\rm l}}\oplus \Hr_{\rm h}=
\Hr_{{\rm l}_0}\oplus \Hr_{{\rm l}_1}\oplus
\Hr_{\rm h} \end{equation}
where the space of \emph{high energy forms} $\Hr_{\rm h}$ is by definition the
orthogonal complement of $\Hr_{{\rm l}}:=\Hr_{{\rm l}_0}\oplus
\Hr_{{\rm l}_1}$. This terminology is justified by the next
proposition. See also \cite{lott} for a similar phenomenon with a
different proof, and \cite{launis,nis} for  related applications of
pseudodifferential operators in determining  essential spectra of
Laplacians. 
  
\begin{proposition}\label{p:free}  
The Laplacian $\Delta_{k}$ on $X'$ stabilizes the decomposition 
\eqref{dp}. Let $\Delta_{k}^{{\rm l}_0}$,
$\Delta_{k}^{{\rm l}_1}$ and $\Delta_{k}^{\rm h}$ be the
Friedrichs extensions of the restrictions of $\Delta_{k}$ to these
spaces, respectively. Then $\Delta_{k}^{\rm h}$ has compact
resolvent, and
\begin{align*}    
\Delta^{{\rm l}_0}_{k}=\big(D^*D+c_0^2x^{2-2p}\big)\otimes 1, & 
& \Delta^{{\rm l}_1}_{k}=\big(D^*D+c_1^2x^{2-2p}\big)\otimes 1,
\end{align*}  
where $c_0, c_1$ are defined by \eqref{c0c1}
and $D$ is the closure of $x^{2-p}\px-c_0x^{1-p}$ with initial domain
$\cC^\infty_c\big((0, \varepsilon )\big)$ in $\Kr$. 
\end{proposition}
\begin{proof}
The de Rham operator on $X'$ stabilizes the orthogonal decomposition 
\eqref{dp}, so the Laplacian $d\delta+\delta d$ does the same. Let $P$
denote the orthogonal  
projection in $L^2(M,\Lambda^k M\oplus\Lambda^{k-1} M)$ onto the 
finite-dimensional space $\ker(\Delta^M_k)\oplus \ker(\Delta^M_{k-1})$ 
of harmonic forms.
Choose a real Schwartz cut-off function $\psi\in\cS(\rz)$ with $\psi(0)=1$. 
Then $\psi(\xi)P$ defines a suspended operator of order $-\infty$
(see e.g., \cite[Section 2]{kso}). From \eqref{cndkp} we see that
$\cN(x^{2p}\Delta_{k})(\xi)+\psi^2(\xi)P$  is strictly positive,
hence invertible for all $\xi\in\rz$. By the surjectivity of the
normal operator, there exists $R\in\Psi_c^{-\infty}(X,\Lambda^k X)$
such that in the decomposition \eqref{decom} over $M$,
$\cN(R)(\xi)=\psi(\xi)P$. Fix $\phi\in\cC^\infty_c(X)$ which equals $1$ on the
complement of $X'$ in $X$, and yet another cut-off function
$\eta$ on $\oX$ which is $1$ near $M$ and such that $\eta\phi=0$.
By multiplying $R$ both to the left and to the right by $\eta$ we can assume that
$R\phi=\phi R=0$, without changing $\cN(R)$. The Schwartz kernel of
$R$ can be chosen explicitly
\begin{equation*}
\kappa_R(x,x',z,z')=\eta(x)\hat{\psi}\left(\frac{x-x'}{x^2}\right)\eta(x')
\kappa_P(z,z')
\end{equation*}
where $\kappa_P$ is the Schwartz kernel of $P$ on $M^2$ and $\hat{\psi}$
is the Fourier transform of $\psi$. Assume now that 
$\hat{\psi}$ has compact support, thus $R$ preserves the space 
$\cunc(X',\Lambda^k X)$.
Let $R_p:=x^{-p}R\in x^{-p}\Psi_c^{-\infty}(X,\Lambda^k X)$.
Then $R_p^*R_p\in x^{-2p}\Psi_c^{-\infty}(X,\Lambda^k X)$
is symmetric on $\cunc(X,\Lambda^k X)$ with respect to
$dg_p$. Moreover  $\Delta_{k}+R_p^* R_p$ is fully elliptic, so 
by \cite[Theorem 17]{wlom}, it is essentially self-adjoint on
$\cunc(X,\Lambda^k X)$ and has purely discrete spectrum. 
Now, noticing that $R$ preserves the decomposition \eqref{dp},
and acts by $0$ on $\Hr_{\rm h}$ and by using the decomposition principle
\cite[Proposition C.3]{GMo1} for $\Delta_{k}$ and
$\Delta_{k}+R_p^* R_p$, we deduce that  
$\sigma_{\rm ess}(\Delta_{k}+R_p^*R_p)=\sigma_{\rm ess}(\Delta^{\rm
h}_{k})=\emptyset$.

For the low energy forms, one gets $\Delta^{{\rm l}_0}_{k}=-x^{(2k-n)p}x^2\px
x^{(n-2k-2)p}x^2\px\otimes 1$, 
acting in $\Hr_{{\rm l}_0}$, and $\Delta^{{\rm l}_1}_{k}=-x^2\px
x^{(2k-2-n)p}x^2\px x^{(n-2k)p}\otimes 1$ acting in $\Hr_{{\rm
    l}_1}$. The proof is finished by expanding 
$D^*D$. \end{proof}   

\begin{proposition}\label{p:thema}
Let $(X,g_p')$ be a Riemannian manifold with metric 
$g_p'$ satisfying the bounds \eqref{e:asympto} and \eqref{e:asympto2}
with respect to  $g_p=x^{2p}g_0$, for some exact cusp metric $g_0$.
\begin{enumerate} 
\item For $0<p\leq 1$, consider the Friedrichs extension  of $\Delta_{k}$. 
Then its essential spectrum is given by $[\inf(\kappa(p)), \infty)$.
\item If $p>1$ and $g_p':=g_p$ is the
unperturbed metric given in \eqref{gp'}, then every self-adjoint
extension of $\Delta_{k}$ has empty essential spectrum.   
\end{enumerate} 
\end{proposition}

In particular, when $X$ is complete (i.e., $p\leq 1$) 
the Laplacian of $g_p$ on forms of degree
$0$ and $1$ always has non-empty essential spectrum. If moreover 
the boundary at infinity $M$ has at least one orientable connected component, 
then the same holds for forms of degrees $n-1$ and $n$.
Note that Theorem \ref{t:Ith} does not follow from Proposition 
\ref{p:thema} since in Section
\ref{lcm} we do not assume the metric to be exact.
\proof We start with the complete case. For a smooth complete metric, 
the essential self-adjointness is a well-known general fact \cite{Ga}. 
As the metric we consider is not smooth, we consider the Friedrichs extension. 
In the exact case, $g_p$ is quasi-isometric to   
the metric \eqref{gp'}. Using Proposition \ref{p:stabess},
in order to compute the essential spectrum we may replace $h(x)$ in
\eqref{cume2} by the metric $h_0:=h(0)$ on $M$, extended to a
symmetric $2$-tensor constant in $x$ near $M$, and we may also set $\rho=0$. 
By \cite{GMo1}[Lemma C.1], the essential spectrum $\siges(\Delta_k)$
is given by  
$\cup_i \siges(\Delta^{\rm l_i}_k)$ on $X'$ of Proposition \ref{p:free}. 
To conclude, remark that the essential spectrum of $D^*D$ is $[0,\infty)$
and the essential spectrum of $\Delta^{\rm l_i}_k$ is $[\lim_{x\rightarrow 0} 
c_i x^{2-2p}, \infty)$ since a bounded potential tending to $0$ is 
a relatively compact perturbation
and thus does not affect the essential spectrum. One may also
compute the essential spectrum using \eqref{e:H_0}.

Let now $p> 1$. The metric is no longer complete so one can not apply
Proposition \ref{p:stabess}. By \cite[Lemma B.1]{GMo1} and 
by the Krein formula, all self-adjoint extensions have the same essential
spectrum so it is enough to consider the Friedrichs extension of
$\Delta_k$.  We now use Proposition \ref{p:free} and \cite[Lemma C.1]{GMo1}. 
The operator $D^*D$ is non-negative, so the spectrum of
$\Delta_k^{\rm{l}_i}$ is contained in  $[\varepsilon^{2-2p}c_i^2,\infty)$. By
\cite[Lemma C.1]{GMo1},  the essential spectrum does not depend  on
the choice of $\varepsilon$.  Now we remark that $p>1$ implies
$\lim_{\varepsilon\to 0}\varepsilon^{2-2p}=\infty$ and also that
$c_i\neq 0$, $i=0,1$ for the constants $c_0,c_1$ defined by \eqref{c0c1}. 
Indeed, the equality $c_i= 0$ would  imply $1/p=\pm (2k\pm2-n)\in\zz$, which
contradicts $p>1$. Thus  by letting $\varepsilon \to 0$ we conclude that the
essential spectrum of $\Delta_k$ is empty.\qed

\section{The Mourre estimate}\label{s:Mourre}
In Mourre theory, one has to construct a conjugate operator in
order to obtain the positivity of a commutator. To this purpose in
Section \ref{s:diago} we write $\Delta_{k}$ near infinity
in terms of the function $L$ defined in equation \eqref{e:L}
instead of the boundary defining function $x$. In Section \ref{s:conj}
we construct the conjugate operator in terms of this new variable. 
Its support is in the cusp, so one
can restrict the analysis there. In Section \ref{s:unper} we prove
the Mourre estimate for the unperturbed metric in Theorem
\ref{t:mourre_0}. Finally, in Section \ref{s:SRLR} we show the main
Theorem \ref{t:Ith1} for the perturbed metric $g_p$ given in
\eqref{gp'}. In this section, we concentrate on the complete case,
i.e.\ $p\leq 1$, otherwise there is no essential spectrum and the
whole analysis becomes trivial (one may take the conjugate operator to be $0$).

\subsection{Diagonalization of the free Laplacian}\label{s:diago}
We now construct a partial isometry. We carry out the analysis on the
cusp $X'$. We start with  \eqref{dp} and work on $\Kr$, see 
\eqref{e:Kr}. We conjugate first through the unitary transformation
\begin{align}\label{ut}
L^2\left( x^{(n-2k)p-2} dx\right)\to L^2\left( x^{p-2} dx\right)&&
\phi\mapsto x^{(n-2k-1)p/2}\phi
\end{align}
Then we proceed with the change of variables 
\begin{align}\label{ut'}
 r:=L(1/x),&& \text{ where $L$ is given by \eqref{e:L}}.
\end{align} 
Therefore, $\Kr$ is unitarily sent into
$\Kr_0:=L^2\big((c,\infty), dr\big)$ for some positive $c$. By tensoring
with the identity on $M$, we have constructed a unitary
transformation:
\begin{eqnarray*}
\cU: L^2\big( (0, \varepsilon )\times M, \Lambda^k X, g_p\big)
\longrightarrow L^2\big( (c, \infty)\times M, \Lambda^k X, dr\,
\mbox{vol}(h)\big)=:\Hr_0.  
\end{eqnarray*} 
We set $X_0':= (c, \infty)\times M$. Note that $\cU$ is an isomorphism
between $\cC^\infty_c( X',\Lambda^*X)$ and
$\cC^\infty_c(X_0',\Lambda^*X)$.  

Set $L_0:= \cU L \cU^{-1}$,  where
the operator of multiplication corresponding to $L$, given by
\eqref{e:L}. We choose $\varepsilon$ small enough so as to get $L_0$ is
the operator of multiplication by $(r,m)\mapsto r$ in $\Hr_0$.  
We consider the closure of $\Delta_{k,0}:= \cU \Delta_k \cU^{-1}$ defined
above $\cC^\infty_c(X_0',\Lambda^*X)$.  Recalling \eqref{dp},
it acts on $\cC^\infty_c
\big((c,\infty)\big) \otimes \cC^\infty(M, \Lambda^k  M\oplus
\Lambda^{k-1} M)$ as  
\begin{align}\label{e:H_0} 
\Delta_{k,0}= \sum_{i=0,1}(-\partial^2_r + V_{p_i})\otimes
P_{\ker(\Delta^M_{k-i})} + \Delta_{k,0}^{\rm h}\otimes P_0^\perp,&&
V_{p_i}(r)=\begin{cases}
c_i^{2} & \text{for $p=1$}\\
a_i/r^2 & \text{for $p<1$}
\end{cases}
\end{align}
for certain $a_i$, where $P_0:=P_{\ker(\Delta^M_{k})}\oplus
P_{\ker(\Delta^M_{k-1})}$, and $c_0$, $c_1$ are defined in
\eqref{c0c1}. Note that $\Delta_{k,0}$ is not self-adjoint. Its
spectrum is $\C$. We will only use it as an auxiliary operator. 

\subsection{The conjugate operator}\label{s:conj}
We now construct a conjugate operator so as to establish a Mourre
estimate for the Laplacian acting on $k$-forms for the free metric
$g=g_p$, given by \eqref{gp'}.  
This section is close to \cite{GMo1}[Section 5.3] for the commutator properties
but the proof of the Mourre estimate differs. We provide full details.
Let $\xi\in\cC^\infty\big((c,\infty)\big)$ such that the
  support of $\xi$ is 
contained in $[3c,\infty)$ and that $\xi(r)=r$ for $r\geq 4c$ and let
$\tilde \cchi\in\cC^\infty\big((c,\infty)\big)$ with support in
$[2c,\infty)$, which equals $1$ on $[3c,\infty)$. By abuse of notation,
we denote $\tilde\cchi\otimes 1$ and $\xi\otimes 1$
in $\cC^\infty(X_0')$ by $\tilde \cchi$ and $\xi$, respectively. 
We also write with the same symbol an operator of multiplication by a
function acting in a Hilbert space and the function.
Choose $\Phi\in\cC^\infty_c(\R)$ with $\Phi(x)=x$
on $[-1,1]$, and set $\Phi_R(x):=R\Phi(x/R)$. We define on
$\cC^\infty_c(X_0', \Lambda^k X)$ a micro-localized version of the
generator of dilations:
\begin{equation}\label{e:SR}
S_{R,0}:= \tilde\cchi \big(\Phi_R(-i\partial_r) \xi +\xi
\Phi_R(-i\partial_r)\big)\otimes P_0 \,\tilde\cchi.
\end{equation}
The operator $\Phi_R(-i\partial_r)$ is defined on the real line
by $\Fr^{-1}\Phi_R\Fr$, where $\Fr$ is the unitary Fourier
transform. We also denote its closure by $S_{R,0}$.
Let $\widetilde\mu_0\in\cC^{\infty}(X_0')$ be with support in $X_0'$ such that 
$\widetilde\mu_0|_{[2c,\infty)\times M}=1$. Set the operator of
multiplication $\widetilde\mu:= \cU^{-1}\widetilde\mu_0\,\cU$. We extend
it by $0$ above the compact part of $X$. On $\cC^\infty_c(X,
\Lambda^k X)$, we set:
\begin{align}\label{e:SRvrai}
S_{R}:= \cU^{-1}S_{R,0}\,\cU\,\widetilde\mu
\end{align}
We denote also by $S_R$ its closure. Note that $S_{R}$ does not
stabilize $\cC^\infty_c(X, \Lambda^k X)$ since $\Phi_R(-i\partial_r)$
acts like a convolution with a function with non-compact support.  However,
$\cC^\infty_c(X, \Lambda^k X)$ is sent into the restriction of the
Schwartz space  $\cU^{-1}\big(\tilde\cchi \Sr(\R)\otimes \im(P_0)\big)$.

By taking $R=\infty$, this operator is also self-adjoint and 
one recovers the conjugate operator initiated in
\cite{FH} for the case of the Laplacian. The drawback of this
operator is that it does not allow very singular perturbation theory
like the one of the metric we consider, see also \cite{GMo1}. Since
here $\Phi_R$ is with compact support, one is able to replace $S_{R}$
by $L$ in the theory of perturbation. We prove some results of
compatibility in the next Lemma. From now on $R$ is finite.

\begin{lemma}\label{l:sa}
For all $R\geq 1$, the operator $S_R$ has the following properties:
\begin{enumerate}
\item  it is essentially
self-adjoint on $\cC^\infty_c(X, \Lambda^k X)$.
\item for every $s,t\in \R^+$, 
$L^{-2}S_{R}^2  \Dc(\Delta_{k}^s)\subset \Dc(\Delta_k^{t})$. \label{p2} 
\item for all $s\in[0,2]$, $\Dc(L^s)\subset \Dc(|S_{R}|^s)$ .
\label{p3}
\end{enumerate}
\end{lemma}
\proof 
We compare $S_R$ with the operator $L$, which is
essentially self-adjoint on $\cC^\infty_c(X, \Lambda^k X)$. Moreover $L$
stabilizes the decomposition \eqref{dp} and we denote by $L_{\rm l}$
its restriction to $\Hr_{\rm l}=\Hr_{\rm   l_0}\oplus \Hr_{\rm l_1}$. 
Set $L_{\rm{l}, 0}:= \cU L_{\rm{l}}\, \cU^{-1}$, which is simply
multiplication by $r$. For short, we write $\Phi_R$ instead of
$\Phi_R(-i\partial_r)$. Take  $\varphi\in \cC^\infty_c(X,\Lambda^k
X)$. Note that $\tilde \varphi:= \cU\, \tilde \mu \varphi \in
\cC^\infty_c(X_0',\Lambda^k X)$.  We have
$S_{R,0} \widetilde \varphi= \tilde \cchi\big(\Phi \xi r^{-1}
- \xi r^{-1}\Phi_R'r^{-1} + \xi r^{-1}\Phi_R \big)\tilde \cchi L_0
\widetilde \varphi$.
Since $\xi r^{-1}$ is bounded in $L^2(X_0,\Lambda^kX)$, we get
$\|S_{R,0} \widetilde \varphi\|\leq a \|L_0 \widetilde\varphi\|$, in
$L^2(X_0', \Lambda^kX)$. Remembering $L\geq 1$, we derive there is $a'$ so
that $\|S_{R} \varphi\|\leq a' \|L \varphi\|$ in $L^2(X,\Lambda^kX)$. 
To be precise, one needs first to approximate $\widetilde \varphi$
with $\widetilde \varphi_\R\otimes (\widetilde \varphi_M+dr\wedge\widetilde
\varphi_M')$, where  $\widetilde \varphi_\R\in\cC^\infty_c(X_0')$, $\widetilde
\varphi_M\in\cC^\infty(M, \Lambda^k M)$ and 
$\widetilde \varphi_M'\in\cC^\infty(M, \Lambda^{k-1} M)$. 

On the other hand, we compute against $\tilde \varphi$ in the form
sense and get
\begin{align*}
[S_{R,0}, L_0] =&
\tilde\cchi\big([\Phi_R , r]\xi  \otimes
P_0 + \xi
[\Phi_R , r] \otimes P_0\big)\tilde\cchi   
=i \tilde\cchi\big(\Phi_R' r r^{-1}\xi  \otimes P_0 + \xi r^{-1}r
\Phi_R'  \otimes P_0\big)\tilde\cchi \\
=&i \tilde\cchi\big(r^{1/2} \Phi_R'  r^{1/2}r^{-1}\xi  \otimes 
P_0 + r^{-1}\xi r^{1/2} \Phi_R'  r^{1/2}  \otimes 
P_0 \big)\tilde\cchi \\
&+i\tilde\cchi\big( [\Phi_R' , r^{1/2}] r^{1/2}r^{-1}\xi  \otimes 
P_0 + r^{-1}\xi  [\Phi_R' , r^{1/2}] r^{1/2}  \otimes 
P_0 \big)\tilde\cchi
\end{align*}
This gives $|\langle S_{R,0} \widetilde \varphi, L_0\widetilde\varphi \rangle
-\langle L_0\widetilde\varphi, S_{R,0}\widetilde\varphi \rangle| \leq
b \|L_0^{1/2}\widetilde\varphi\|^2$ in $L^2(X_0', \Lambda^kX)$. Since
$L\geq 1$, we infer  $|\langle S_R \varphi,
L\varphi \rangle -\langle L\varphi, S_{R}\varphi
\rangle| \leq b' \|L^{1/2}\varphi\|^2$ in $L^2(X,
\Lambda^k X)$. Finally, we conclude
that $S_R$ is essentially self-adjoint on $\cC^\infty_c(X, \Lambda^k X)$
by using \cite[Theorem X.37]{RS2}.

We turn to point (2). By interpolation, it is enough to show that
$L^{-2} S_R^2  (\Delta_k+1)^s$ and $S_R^2L^{-2}(\Delta_k+1)^s$,
defined on $\cC^\infty_c(X, \Lambda^kX)$, extend to
bounded operators in $L^2(X, \Lambda^kX)$, for all $s\in\N$.
The treatment being similar, we deal only with the first term.

Since $\cchi$ stabilizes all the Sobolev spaces in
$\R$, using \eqref{e:H_0} we obtain there is $c$ so that
\begin{eqnarray}\label{e:Deltachange}
\|(\Delta_\R+i)^{-s} \otimes P_0\tilde \cchi (\Delta_{k,0}+i)^s
\widetilde \varphi\| \leq c\| \widetilde \varphi \|,
\end{eqnarray} 
for all $\widetilde\varphi$ taken like above. Here $\Delta_\R$ denotes
the positive Laplacian in $L^2(\R)$. Now remark that thanks to the
cut-off functions $\tilde \cchi$, there is $g\in\cC^\infty_c(\R)$
such that
\begin{eqnarray}\label{e:reecrit}
S_{R,0}= \tilde\cchi \big(2r \Phi_R +i\Phi_R' + g \Phi_R+ \Phi_R g
\big)\otimes P_0 \, \tilde\cchi.
\end{eqnarray} 
Then thanks to \eqref{e:Deltachange} and by going back to $\varphi$ as
above, the result follows if
\begin{eqnarray}\label{e:carre}
r^{-2} \big(2r \Phi_R +i\Phi_R' + g \Phi_R+ \Phi_R g
\big)\tilde\cchi^2 \big(2r \Phi_R +i\Phi_R' + g \Phi_R+ \Phi_R g
\big) (\Delta_\R+1)^{s}
\end{eqnarray} 
defined on $\cC^\infty_c(\R)$ extends to a bounded operator. This
follows easily by commuting all the $r$ to the left and by using the
stability of Sobolev spaces under the  multiplication by a smooth
function with bounded derivatives. 

We focus now on point (3). Using \eqref{e:carre}, we have shown in
particular that  $S_{R}^2L^{-2}$ is bounded in $L^{2}(X,
\Lambda^kX)$. We infer $\|S_{R}^2\varphi\|^2 \leq c
\|L^2\varphi\|$ for all $\varphi \in \cC^\infty(X, \Lambda^k
X)$. Taking a Cauchy sequence, we deduce $\Dc(L^2)\subset
\Dc(S_{R}^2)$. An argument of interpolation concludes.\qed 

\subsection{The Mourre estimate for the unperturbed metric}\label{s:unper}
The main result of this section is the following Mourre estimate for
the free operator $\Delta_{k}$. The proof is more involved than in
\cite{GMo1} although the computations of commutators are essentially the same. 
The problem comes from the fact that $\Delta_{k}$ has
\emph{two} thresholds, not only one like the magnetic Laplacian acting on
functions. Above the two thresholds we can pursue the same analysis as
\cite{GMo1}. Morally speaking, as in $3$-body problems, 
when we localize the energy between the
thresholds, the smallest one will yield the positivity of the commutator
and the largest one will bring more compactness. 
The technical problem
comes from the fact that the resolvent of $\Delta_{k}$, 
and therefore the spectral measure, does not
stabilize the decomposition \eqref{dp}. To have such a decomposition, 
one would have to uncouple the cusp part from the compact part. Here
some caution should be exercised, since considering the Friedrichs
extension of $\Delta_{k}$, on the  cusp and on the compact part,
would be too singular  a perturbation, even though it is enough for
the study of the essential spectrum,  see \cite[Lemma C.1]{GMo1}.  
We uncouple only the low energy part.   

We recall the regularity classes of the Mourre theory. We refer to   
\cite{ABG} for a more thorough discussion of these matters. Take $H$
and $A$ two self-adjoint operators acting in a Hilbert space $\Hr$. We
say $H\in \cC^k(A)$ if $t\mapsto e^{-itA}(H+i)^{-1} e^{itA}$ is
strongly $\cC^k$ in $\Bc(\Hr)$. Moreover, if $e^{itA}\Dc(H)\subset
\Dc(H)$, this is equivalent to the fact that $H\in \cC^k\big(A; \Dc(H),
\Dc(H)^*\big)$, i.e.\ the function $t\mapsto e^{-itA}H e^{itA}$ is
strongly $\cC^k$ in $\Bc\big(\Dc(H), \Dc(H)^*\big)$. Recall that we use the
Riesz isomorphism to identify $\Hr$ with $\Hr^*$. We now give the main result 
of this section.

\begin{theorem}\label{t:mourre_0}
Let  $R\geq 1$ and $p\leq 1$. Given an interval $\cJ$ which does not
contain $\kappa(p)$, see \eqref{e:kappa},
let $c_\cJ < d\big(\inf(\cJ), \{c\in \kappa(p), c \leq
\inf(\cJ)\}\big)$. We have that $e^{itS_{R}} \Dc(\Delta_{k})\subset
\Dc(\Delta_{k})$ and $\Delta_{k}\in\cC^{2}\big(S_{R};
\Dc(\Delta_{k}), \Hr\big)$. Moreover, there exist $\varepsilon_R> 0$ and a compact
operator $K_R$ such that the inequality
\begin{equation}\label{e:mourre_0}
E_\cJ(\Delta_{k}) [\Delta_{k},iS_{R}] E_\cJ(\Delta_{k}) \geq (4 c_\cJ
 - \varepsilon_R)
E_\cJ(\Delta_{k})+ K_R \end{equation} 
holds in the sense of forms, and such that $\varepsilon_R$ tends to $0$
as $R$ goes to infinity.
\end{theorem} 

We now go in a series of Lemmata and prove this theorem at the end of
the section. 

\begin{lemma}\label{l:commu}
The commutators $[\Delta_{k},iS_{R}]$ and
$\big[[\Delta_{k},iS_{R}],iS_{R}\big]$, taken in the form sense on
$\cC^\infty_c(X, \Lambda^{k}X)$, extend to bounded operators in $\Hr$. 
\end{lemma}

\proof Let $\varphi\in \cC^\infty_c(X, \Lambda^k X)$. Set $\widetilde
\varphi:= \cU\, \tilde \mu \varphi$ and write $\Phi_R$ instead of
$\Phi_R(-i\partial_r)$. To justify the computations in the tensor
product form in $L^2(X_0', \Lambda^k X)$ we approximate $\widetilde
\varphi$ like in Lemma \ref{l:sa}.  

For instance, for the commutator  $[\Delta_{k}, S_{R}]$, we use
the transformation $\cU$. We show that
$|\langle \widetilde\varphi, [\Delta_{k,0}, S_{R,0}]
\widetilde\varphi\rangle|\leq c\|\widetilde \varphi\|^2$. Thanks to
the support of the commutator, we can take off  $\tilde\mu$ and
infer $|\langle \varphi, [\Delta_{k}, S_{R}] \varphi\rangle|\leq
c'\|\varphi\|^2$. To lead the computation, we use the expression
obtained in \eqref{e:H_0}. The high energy part plays no r\^ole. 
We drop $P_0$ for convenience and it is enough to compute in the form sense in
$\cC^\infty_c(\R)$.

We first deal with the part in $\partial_r$. We now rewrite $S_{R,0}$
like in \eqref{e:reecrit}.   By commuting till all the $\partial_r$
are next to a $\Phi_R$, we get  
\begin{equation}\label{e:commu5}\begin{split}
[\partial_r^2, \tilde\cchi \Phi_R \xi +\xi
\Phi_R \tilde\cchi]&= 
[ \partial_r^2, \tilde\cchi \big(2r \Phi_R +\Phi_R' + g \Phi_R+ \Phi_R g
\big)\tilde\cchi]= 4 \partial_r\Phi_R + \mbox{bounded}
\end{split}\end{equation}
In the same way, we obtain
$\big[[\partial_r^2, \tilde\cchi \Phi_R \xi +\xi
\Phi_R \tilde\cchi], \tilde\cchi \Phi_R \xi +\xi
\Phi_R \tilde\cchi\big]=  8 \partial_r\Phi_R + \mbox{bounded}$.

For $p<1$, the potential part $V_{p_i}$ arises. We extend it 
smoothly in $\R$ and with compact support in $\R^{-}$. 
We treat its first commutator: 
\begin{equation}\label{e:commu6}\begin{split}
[V_{p_i}, \tilde\cchi \Phi_R \xi +\xi \Phi_R\tilde\cchi]
&=2\tilde\cchi[V_{p_i},\Phi_R]r\tilde\cchi+\mbox{bounded.}
\end{split}\end{equation}
Now develop the commutator and commute until $r$ touches $V_{p_i}$ to conclude. 

Consider finally the second commutator of $V_{p_i}$. We treat only the most
singular part
$\big[[V_{p_i}, \Phi_R]r, \Phi'_R r\big]
= \big[[V_{p_i}, \Phi_R], \Phi_R \big]r^2 + i
\Phi_R [\Phi_R', V_{p_i}]r + [V_{p_i}, \Phi_R] \Phi_R'' -i [V_{p_i},
\Phi_R]r \Phi_R'$. As above commute till $r$ touches $V_{p_i}$ so as
to get the boundedness. 
\qed

As pointed out in \cite{ggm}, the $\cC^1$ assumption of regularity is
essential in the Mourre theory and  should be checked carefully.
See \cite{Haefner} for a different approach. 
\begin{lemma}\label{l:regu0}
For $R\geq 1$, one has $\Delta_{k}\in\cC^1(S_{R})$ and 
$e^{itS_{R}} \Dc(\Delta_{k}) \subset\Dc(\Delta_{k})$. 
\end{lemma}
\proof
We check the hypothesis of Lemma \cite{GMo1}[Lemma A.2] so as to get
the $\cC^1$ property. Let 
$\cchi_n(r):=\cchi(r/n)$ and $\Dr:=\cC^\infty_c(X, \Lambda^k X)$. Since
$\|\cchi_n'\|_\infty$ and $\|\cchi_n''\|_\infty$ are uniformly bounded, we get 
that $\sup_n\|\cchi_n\|_{\Dc(H)}$ is finite. By lemma \ref{l:sa},
$\Dr$ is a core for  
$S_{R}$. Assumption (1) is obvious, assumption (2) holds since $(1-\cchi_n)$ has
support in $[2n,\infty)$ and assumption (3) follows from the fact that
$H$ is elliptic, so the resolvent of $\Delta_{k}$ sends $\Dr$ into
$\cC^\infty(X, \Lambda^k X)$. The point (A.6) follows from Lemma
\ref{l:commu}. It remains to show (A.5). Let  
$\phi\in\cC^\infty(X, \Lambda^k X)\cap\Dc(\Delta_{k})$ and set
$\tilde \phi:= \cU\, \tilde \mu \phi$. 
We get $iS_{R,0}[\Delta_{k,0},\cchi_n]\tilde\phi=
\tilde\cchi(2\Phi_R(\partial_r) + [\xi, \Phi_R(\partial_r)]
\xi^{-1})\tilde \cchi P_0 \big(2\xi\cchi_n'\partial_r \tilde\phi +
\xi\cchi_n'' \tilde\phi\big)$. Both terms are tending to $0$ because 
$\supp(\cchi_n')\subset[n,2n]$ and $\xi \cchi_n^{(k)}$ tends
strongly to $0$ on $L^2(\R^+)$ for $k\geq 1$. Note one may remove
$\widetilde \mu$ because of the support. This shows $\Delta_k\in\cC^1$.

Since $[\Delta_{k},
iS_{R}]\in\Bc\big(\Dc(\Delta_k),\Hr\big)$, \cite[Lemma 2]{GG0} concludes 
$e^{itS_R}\Dc(\Delta_k)\subset\Dc(\Delta_k)$. \qed  

\begin{lemma}\label{l:regu}
Let $R\geq 1$. Then $\Delta_{k}$ belongs to $\cC^2\big(S_{R};
\Dc(\Delta_k),\Hr\big)$ for  $p\leq 1$. 
\end{lemma} 
\proof
From Lemma \ref{l:regu0}, we have the invariance of the domain 
and from Lemma \ref{l:commu}, the two commutators of $\Delta_k$ with $S_{R}$ extend to bounded 
operators. \qed

We now introduce an intermediate operator, to be able to deal with the
range between the thresholds when both Betti numbers are non-zero. 
This situation is somewhat similar to adding an anisotropic potential
for the Euclidean Laplacian in dimension $1$.
Let $\dot\Delta_{k}$ be Friedrichs extension of $\Delta_{k}$
over 
\begin{equation}\label{e:tildeH_0} 
\Dr_{\rm l}=\{f\in \cC^\infty_c(X, \Lambda^k X), P_0f(x,m)=0, \text{ for $x=\varepsilon$ and  
for all $m\in M$}.  
\end{equation}
Here $\varepsilon$ is chosen as in Section \ref{s:diago}.
We note that $\Delta_{k}|_{\Dr_{\rm l}}$ has finite deficiency
indices since $P_0$ is of finite rank by compactness of $M$. Using the
Krein formula and the Stone-Weierstrass  theorem, we infer 
\begin{equation}\label{e:tildeH_02}
\theta(\Delta_{k}) -\theta\big(\dot\Delta_{k}\big) \in
\Kc(\Hr), \text{ for all $\theta\in\cC_c(\R,\C)$.}
\end{equation} 
The main interest of $\dot\Delta_{k}$ is that its
resolvent stabilizes the low energy part of the decomposition \eqref{dp}. 
Note the high energy part interacts with the compact part of $X$.

\begin{remark}\label{r:regu} We have
$\dot\Delta_{k}\in\cC^1(S_{R})$. In fact, because of the
support of $\tilde \cchi$,  $e^{itS_{R,0}}$ acts like
identity on the orthogonal complement of $L^2\big([c,\infty)\big)\otimes
L^2(M)$ in $\Hr_0$. Combining Lemma  \ref{l:regu} and
interpolation, we deduce that $e^{itS_{R}}$ stabilizes the
form domain of $\dot\Delta_{k}$. Then, repeating Lemma
\ref{l:commu} and using \eqref{e:tildeH_02}, we infer
$\dot\Delta_{k}\in \cC^1\big(S_{R};
\Dc(\dot\Delta_{k}^{1/2}),
\Dc(\dot\Delta_{k}^{1/2})^*\big)$. Note that $[\dot\Delta_{k},
i S_R]$ is equal to the operator $[\Delta_{k}, i S_R]$ thanks to the support of $S_R$.
\end{remark} 

We are now in position to prove the main result of this section.

\proof[Proof of Theorem \ref{t:mourre_0}] The regularity assumptions
follow from Lemmata \ref{l:regu0} and \ref{l:regu}. 
Set $\varphi\in\cC^\infty_c(X, \Lambda^k X)$. Set $\widetilde
\varphi:= \cU\, \tilde \mu \varphi$ and write $\Phi_R$ instead of
$\Phi_R(-i\partial_r)$. To justify the computations in the tensor
product form in $L^2(X_0', \Lambda^k X)$ we approximate $\widetilde
\varphi$ like in Lemma \ref{l:sa}. Because of supports, one has  $\langle \varphi, 
[\Delta_{k},iS_{R}] \varphi\rangle= \langle\widetilde \varphi,
[\Delta_{k,0},iS_{R,0}] \widetilde\varphi\rangle=\langle\widetilde \varphi,
[\dot\Delta_{k,0},iS_{R,0}] \widetilde\varphi\rangle$. 

We concentrate on the low energy part. We add \eqref{e:commu5} and 
\eqref{e:commu6} and infer  
\begin{align}\nonumber
\langle \widetilde \varphi, [\dot\Delta_{k,0},iS_{R,0}]
\widetilde\varphi\rangle =& 
\langle \widetilde \varphi, [\Delta_{k,0},iS_{R,0}]
\widetilde\varphi\rangle
\\\label{e:Mstep1'}
=&\sum_{i=0,1}\langle \widetilde \varphi, 4( -\partial_r^2
+V_{p_i}-V_{p_i}(\infty)+T_R)\otimes
P_{\ker(\Delta^M_{k-i})}\widetilde \varphi \rangle + \langle
\varphi, K_2 \varphi \rangle  
\end{align}
for a certain $K_2=K_2(R)\in\Kc(\Dc(\Delta_{k}),\Dc(\Delta_{k})^*)$ 
and with $T_R:= \partial_r(\partial_r+i\Phi_R)$. The compactness of
$K_2$ follows by noticing that $L_0^{-1}\in
\Kc(\Dc(\Delta_{k,0}),\Hr_0)$ and that  $L_0[V_{p_i}, i S_{R,0}]\in
\Bc(\Hr_0,\Hr_0)$.  
We now control the size of $T_R$. We have
\begin{equation}\label{e:vareR}
\|\widetilde\mu_0 T_R
(1\otimes P_0)\widetilde\mu_0 \|_{\Bc(\Dc(\Delta_{k,0}),
  \Dc(\Delta_{k,0})^*)} \text{ tends 
   to $0$ as $R$ goes to infinity.}
\end{equation} 
Indeed, $\widetilde\mu$ stabilizes $\Dc(\Delta_{k,0})$,
then $-\partial_r^2 \widetilde\mu$ belongs to
$\Bc\big(\Dc(\Delta_{k,0}), L^2(\R)\big)$ and
$(-\partial_r^2+i)^{-2}T_R$ tends to $0$ in norm by functional
calculus, as $R$ goes to infinity.  

We now extract some positivity. 
Let $\theta\in\cC^\infty_c(\R,\R)$ such that
$\theta|_\cJ=1$, such that $\kappa(p)$ is disjoint from the support of 
$\theta$ and such that the distance from $\cJ$ to the complementary of 
the support is smaller than $c_\cJ$. 
If the support is lower than $\kappa(p)$, the Mourre estimate is trivial 
because $\theta(\Delta_{k})$ is compact, by Proposition \ref{p:thema}. 
We assume then that $\supp(\theta)$ is above $\inf\big(\kappa(p)\big)$. 
Using
\eqref{e:tildeH_02}, the fact that $\cU\theta(\dot\Delta_{k})\cU^{-1}$
stabilizes the decomposition \eqref{dp}, that
$[\Delta_{k,0},iS_{R,0}]$ is bounded, we get there is a compact
operator $K_0$ in $\Hr_0$ so that:   
\begin{align*}
\tilde \mu_0
\cU^{-1}\theta(\Delta_{k})\cU [\Delta_{k,0},iS_{R,0}]\cU
\theta(\Delta_{k,0}) \cU^{-1}\tilde \mu_0 =&
\\&\hspace{-3cm}
\tilde \mu_0 P_0\,
\cU^{-1}\theta\big(\dot\Delta_{k}\big)\cU \big[\dot\Delta_{k,0},iS_{R,0}\big] \cU
\theta\big(\dot\Delta_{k}\big)\cU^{-1} P_0\tilde \mu_0 +K_0.   
\end{align*}
By ellipticity, we have $\cU\theta\big(\dot \Delta_{k} \big)
\cU^{-1}\tilde \mu P_0  \Hr_0\subset \cC^\infty_0\big((c, \infty)
\big)\otimes P_0 L^2(M, \Lambda^k M\oplus \Lambda^{k-1} M) $, i.e.
the radial part tends to 
$0$ in $c$ and at infinity. The analysis is reduced to the one of
the Friedrichs extension $H_i^F$ of $H_i=-\partial_r^2+V_{p_i}$ on
$L^2(c, \infty)$. Indeed, by stability of the decomposition
\eqref{dp}, we have $\cU\theta\big(\dot\Delta_{k}\big)\cU^{-1} (1\otimes 
P_0)\tilde \mu_0= \sum_{i=1,2} \theta(H_i^F)\otimes
P_{\ker(\Delta^M_{k-i})} \tilde \mu_0$. 

Up to $T_R$ and compact terms, the commutator of $H_i^F$ with $S_{R,0}$,
given by  \eqref{e:Mstep1'}, is $4H_i^F -V_{p_i}(\infty)$. We have two
possibilities: if the support of $\theta(H_i^F)$ is under
$V_{p_i}(\infty)$ then for all $c\in\R$, there is a compact operator
$K$ such that $\theta(H_i^F)H_i^F\theta(H_i^F) = c \theta(H_i^F) + K$, since
the spectral measure is compact. If the support is above, 
then there
is a compact $K$ such that $\theta(H_i^F)\big(H_i^F
-V_{p_i}(\infty)\big)\theta(H_i^F)\geq \inf\big(\supp(\theta)-
V_{p_i}(\infty)\big)\theta(H_i^F)+K$. Adding back $T_R$ and the compact
part, we obtain:
\begin{align*}
 \tilde \mu_0 P_0\cU \theta(\dot\Delta_{k})\cU^{-1}
[\Delta_{k,0},iS_{R,0}]\cU\theta(\dot\Delta_{k}) \cU^{-1}
P_0\tilde \mu_0 \geq& \, 4c_\cJ \tilde \mu_0 P_0\,
\cU\theta^2(\dot\Delta_{k}) \cU^{-1} P_0\tilde \mu_0
\\ & \hspace{-2cm}
+ \tilde \mu_0 P_0\,\cU\theta(\dot\Delta_{k})\cU^{-1} T_R
\cU\theta(\dot\Delta_{k})\cU^{-1}P_0\tilde \mu_0 + K_0 
\end{align*} 
We now remove the $P_0$. We add $P_0^\perp\theta\big(\dot  
\Delta_k\big)\tilde\mu_0$ on the right and on the left. 
This is a compact operator as it is equal to
\begin{eqnarray*}
(\Delta_{k,0}^F P_0^\perp+i)^{-1}(\Delta_{k,0}^F
  P_0^\perp+i)\tilde\mu_0 P_0^\perp\cU\theta\big(\dot
  \Delta_k\big)\cU^{-1} \tilde\mu_0 +
P_0^\perp\mu_0 \cU\theta\big(\dot  \Delta_k\big)\cU^{-1} \tilde\mu_0,
\end{eqnarray*}
where $\Delta_{k,0}^F P_0^\perp$ denotes the Friedrichs extension of
$\Delta_{k,0} P_0^\perp$. Its resolvent is compact by Proposition
\ref{p:free}. Hence, the first term is compact as  $\mu_0
P_0^\perp\cU\theta\big(\dot  \Delta_k\big)\cU^{-1} \tilde\mu_0$ is
with value in $\Dc\big(\Delta_{k,0}^F P_0^\perp\big)$. The compacity
of the second terms follows from Rellich-Kondrakov. 

We get rid of the term in $T_R$ and let appear $\varepsilon_R$ by
using the bound \eqref{e:vareR}.  Going back to the spectral measure
of $\Delta_{k}$ with the help of \eqref{e:tildeH_02}, we infer: 
\begin{align*}
\tilde \mu_0\,\cU \theta(\Delta_{k}) \cU^{-1}[\Delta_{k,0},iS_{R,0}]
\cU\theta(\Delta_{k})\cU^{-1} \tilde \mu_0 \geq  
4(c_\cJ-\varepsilon_R)
\tilde \mu_0\,\cU\theta^2(\Delta_{k})\cU^{-1}\tilde \mu_0  +K_0
\end{align*} 
holds true in $\Hr_0$ in the form sense. By Rellich-Kondrakov, up to 
more compactness, this implies \eqref{e:mourre_0} with $\theta$ 
instead of $E_\cJ$. To conclude, apply $E_\cJ(\Delta_k)$ on both sides.
\qed

\subsection{The spectral and scattering theory}\label{s:SRLR}
Applying directly the Mourre theory with Theorem \ref{t:mourre_0}, one
obtain the results $(1)-(4)$ of Theorem \ref{t:Ith1} with the metric
$g$. The aim of the section is to complete the result and to go into
perturbation theory so as to get the other metrics. In the earlier version of the Mourre  theory for
one self-adjoint operator $H$ having a spectral gap, one used something stronger than the
hypothesis $H\in\cC^2(A)$, where $A$ the conjugate operator. This means
$[[(H+i)^{-1}, A],A]$ extends to a bounded operator. Keeping this
hypothesis leads to  a too weak perturbation theory. In this section,
we check a weak version of the  two-commutators hypothesis. We say
$H\in \cC^{1,1}(A)$ if  
\begin{equation*}
	\int_0^1 \big\|[[(H+i)^{-1}, e^{itA}], e^{itA}]\big\|\,
\frac{dt}{t^2} <\infty.
\end{equation*}
If we have the invariance of the domain, $e^{itA}\Dc(H)\subset
\Dc(H)$, \cite{ABG}[Theorem 6.3.4.] allows to reformulate this condition
in the equivalent 
\begin{equation*}
\int_0^1 \big\|[[H, e^{itA}],
  e^{itA}]\big\|_{\Bc(\Dc(H), \Dc(H)^*)}\, \frac{dt}{t^2} <\infty
\end{equation*}
 This one is naturally named $\cC^{1,1}(A; \Dc(H), \Dc(H)^*)$. We refer to
\cite{ABG} for more properties. This is the optimal class of operators
which give a limit absorption principle for $H$ in  some optimal Besov
spaces associated to $S_{R}$. For the sake of simplicity, we will not
formulate the result with these Besov space  and keep power of
$\langle S_{R}\rangle^{s}$. Those one are replaced in the final result
by $\langle L\rangle^{s}$ freely, using Lemma \ref{l:sa}.

The proof of Theorem \ref{t:Ith1} is given in the end of the
section. The operator $\Delta_k$ belongs to
$\cC^2(S_R)$ and therefore  to $\cC^{1,1}(S_R)$. We have the
invariance of the domain, by Lemma \ref{l:regu0}.  
Take now a symmetric relatively compact perturbation
$T$ of $\Delta_{k}$ which lies in $\cC^{1,1}(S_R; \Dc(\Delta_k),
\Dc(\Delta_k)^*)$. Then a Mourre estimate 
holds for $\Delta+T$, and one gets \eqref{a)}-\eqref{c)} of Theorem
\ref{t:Ith1} by using 
Theorem \cite[Theorem 7.5.2]{ABG}. The H\"older regularity of the
resolvent, point \eqref{d)}, follows from \cite{ggm}, see references
therein. If the perturbation is short-range (see below) or local, then
one has point \eqref{e)} by using \cite[Theorem 7.6.11]{ABG}. 

We keep the notation of section \ref{s:diago}. Let 
$\partial_L$ be the closure of the operator defined by
\begin{align*}
(\partial_L f)(x, m)= x^{2-p}(\partial_x f)(x)- \alpha_k x^{1-p}f(x),&&
\text{ with } \alpha_k=(n-2k-1)p/2,
\end{align*} 
for $f\in \cC^\infty_c(X', \Lambda^k X)$. Note that $\partial_r=\cU \partial_L\cU^{-1}$. 
Consider a  symmetric differential operator $T:
\Dc(\Delta_{k})\rightarrow  
\Dc(\Delta_{k})^*$. Let $\theta_{\rm sr}$ be in $\cC_c^\infty((0,\infty))$
not identically $0$; $T$ is \emph{short-range} if
\begin{equation}\label{e:sr}
\int_1^\infty \Big\|\theta_{\rm sr}\Big(\frac{L}{r}\Big) T
\Big\|_{\Bc(\Dc(\Delta_{k}), \Dc(\Delta_{k})^*)} dr <\infty.
\end{equation} 
and \emph{long-range} if
\begin{equation}\label{e:lr}\begin{split}
\int_1^\infty  \Big\|[T, L]\theta_{\rm lr}\Big(\frac{L}{r}\Big)
\Big\|_{\Bc(\Dc(\Delta_{k}), \Dc(\Delta_{k})^*)}  
+  \Big\|\tilde\mu[T, P_0]L\theta_{\rm
lr}\Big(\frac{L}{r}\Big)\tilde\mu\Big\|_{\Bc(\Dc(\Delta_{k}),
\Dc(\Delta_{k})^*)} 
\\
+  \Big\|P_0\tilde\mu[T, \partial_L]L\theta_{\rm
lr}\Big(\frac{L}{r}\Big)\tilde\mu \Big\|_{\Bc(\Dc(\Delta_{k}),
\Dc(\Delta_{k})^*)} \frac{dr}{r} <\infty.
\end{split}\end{equation}
where $\theta_{\rm lr}$ is the characteristic function of $[1,\infty)$
in $\R$.

The first condition is evidently satisfied if there is $\varepsilon
> 0$ such that  
\begin{equation}\label{e:sr'}
	\|L^{1+\varepsilon} T\|_{\Bc(\Dc(\Delta_{k}),
\Dc(\Delta_{k})^*)} < \infty
\end{equation}
and the second one if
\begin{equation}\label{e:lr'}\begin{split}
\|L^{\varepsilon} [T,L]\|_{\Bc(\Dc(\Delta_{k}), \Dc(\Delta_{k})^*)}	
+
\|L^{1+\varepsilon}
\widetilde\mu[T,P_0]\widetilde\mu\|_{\Bc(\Dc(\Delta_{k}),
\Dc(\Delta_{k})^*)} 
\\
+ \|L^{1+\varepsilon}
\widetilde\mu[T,\partial_L]P_0\widetilde\mu\|_{\Bc(\Dc(\Delta_{k}),
\Dc(\Delta_{k})^*)}< \infty  
\end{split}\end{equation}
The condition involving $[P_0, T]$ essentially tells us 
that the non-radial part of $T$ 
is a short-range perturbation. This is why we will ask the long-range
perturbation to be radial. To show that the first class is in 
$\cC^{1,1}(S_R)$, one uses \cite[Theorem 7.5.8]{ABG}. 
The hypotheses are satisfied thanks to Lemmata \ref{l:sa} and \ref{l:L}. 
Concerning the second class, it is enough to show that
$[T,S_R]\in\cC^{0,1}(S_R)$; this follows by using \cite[Proposition
7.5.7]{ABG} (see the proof of \cite[Proposition 7.6.8]{ABG} for instance). 

Straightforwardly, we get the short and long-range perturbation of the
potential.  

\begin{lemma}\label{l:V}
Let $V\in L^\infty(X)$. If $\|L^{1+\varepsilon}V\|_\infty<\infty$ then 
the perturbation $V$ is short-range in $\Bc(\Hr)$. If $V$ is radial, 
tends to $0$ as $x\to 0$ 
and $\|L^{1+\varepsilon} x^{2-p}\partial_x V\|_\infty< \infty$,
then $V$ is a long-range perturbation in $\Bc(\Hr)$. 
\end{lemma} 

In this section we will prove  the points \eqref{b)} -- \eqref{e)} of Theorem
\ref{t:Ith1} for the perturbed metric $\tilde g:=(1+\rho_{\rm sr}+\rho_{\rm
lr})g_p$, where $\rho_{\rm sr}$ and $\rho_{\rm lr}$ are in
$\cC^\infty(X)$,  $\inf_{x\in X}
\big(1+\rho_{\rm sr}(x)+\rho_{\rm lr}(x)\big)>1$, $\rho_{\rm lr}$ is radial, where $g_p$ is given by
\eqref{gp'} and where we assume
\begin{equation}\label{e:SRLR}
\begin{split}
\|L^{1+\varepsilon}\rho_{\rm sr}\|_\infty+\|d\rho_{\rm sr}\|_\infty +
\|\Delta\rho_{\rm sr}\|_\infty<\infty, &\\
 \|L^{\varepsilon}\rho_{\rm lr}\|_\infty+\|L^{1+\varepsilon}
x^{2-p}\partial_x \rho_{\rm lr}\|_\infty + \|\Delta\rho_{\rm lr}\|_\infty<\infty.&
\end{split}
\end{equation} 
Lemma \ref{l:uni} shows  the Laplacian associated to the metric $\rho_{\rm sr}$ ($\rho_{\rm lr}$ resp.)
is a short-range perturbation (long-range resp.) of the one given by the metric $g$ in some precise sense.
\begin{remark} For the sake of simplicity, we treat 
only conformal perturbations of the metric. However,
there is no obstruction in our method to 
treat more general perturbations like in \cite{FH}. The only fact we really use
is that the domain of the perturbed Laplacian is the same as that of the free Laplacian. In the general
case we need to impose an extra condition \eqref{e:asympto2}. In the case
of a conformal change, this follows automatically by the condition \eqref{e:SRLR} on the derivatives. 
\end{remark} 

We define the Sobolev spaces associated to the Laplacian given by the metric $g$. Set $\Hr^s:=\Dc\big((1+\Delta)^{s/2}\big) = \Dc\big(|d+\delta|^s\big)$ for 
$s\geq 0$, and $\Hr^s:=(\Hr^{-s})^*$ for $s$ negative (here we
use the Riesz isomorphism between $\Hr:=\Hr^0$ and $\Hr^*$). We will use $\Hr^s$ with $s\in[-2,2]$;
these spaces are equal to the ones obtained using the metric $\tilde g$ since the norms are equivalent.

To analyze the perturbation of the metric, unlike in
\cite{GMo1}, we work simultaneously with forms of all degrees. 
Since $L$ stabilizes the space
of $k$-forms, one gets $\cC^{1,1}$ regularity for $\Delta_k$ by
showing it for $\Delta$. 

We now proceed in two steps and set $g_{\rm lr}:=(1+\rho_{\rm lr})g$. 
We define $U_{\rm lr}$, $U_{\rm sr}$ being the operators of multiplication by 
$(1+\rho_{\rm lr})^{(n-2k)/4}$ and $\big(1+\rho_{\rm sr}/(1+\rho_{\rm lr})\big)^{(n-2k)/4}$, respectively, when restricted to $k$-forms. Remark that $\rho_{\rm sr}/(1+\rho_{\rm lr})$ is short range in the 
sense of \eqref{e:SRLR}. Now note that the transformation 
\begin{align*}
U_{\rm lr}:& L^2(X, \Lambda^* X, g) \longrightarrow L^2(X, \Lambda^* X, g_{\rm lr}):=\Hr_{\rm lr}
\\
U_{\rm sr}:& L^2(X, \Lambda^* X, g_{\rm lr}) \longrightarrow L^2(X, \Lambda^* X, \tilde g):=\Hr_{\tilde g}
\end{align*}
are unitary. Set also $U_{\tilde g}:= U_{\rm sr}U_{\rm lr}$.

\begin{lemma}\label{l:U}
The operators $L^{1+\varepsilon}(U_{\rm sr}-1)$, $L^{\varepsilon}(U_{\rm lr}-1)$
and $L^{1+ \varepsilon}[U_{\rm lr}, \partial_L]$ extend to bounded operators in $\Hr$.
\end{lemma}
\proof Concentrate on the short-range case. Easily, it is enough to show that,  for all non-negative
integer $\alpha$, $L^{1+\varepsilon}\big((1+\rho_{\rm sr}/(1+\rho_{\rm
  lr}))^\alpha -1\big)$ is bounded. One shows that using
\eqref{e:SRLR} and the binomial formula. For the first part of the
long-range case, we do the same. The second part is obvious using the
chain-rule and again \eqref{e:SRLR}. \qed

We gather various technicalities concerning the operator $L$. 
By taking the adjoint and as the spaces $\Hr^s$ are independent of
$\rho_{\rm sr}$ and $\rho_{\rm lr}$, one gets similar bounds for
$\delta_g, \delta_{g_{\rm lr}}$ and $\delta_{\tilde g}$.  
Note we add the metric in subscript to stress the dependence of the metric.

\begin{lemma}\label{l:L}
Let $s\in[-2,2]$ and $\alpha\geq 0$. 
\begin{enumerate}
\item The operator of multiplication by $dL$ is bounded in $\Hr$.
The support of $dL$ is in $(0, \varepsilon)\times M$ and
$dL= -x^{p-2} dx$ for $x$ small enough. 
\item $e^{itL}\Hr^2\subset \Hr^2$ and
  $\|e^{itL}\|_{\Bc\left(\Hr^2\right)}\leq c  
(1+t^2)$. 
\item $L^{-\alpha }[d,L^{1+\alpha}]$ is bounded in $\Hr$ and
$L^{-\alpha }[\Delta_{\tilde g}, L^{1+  \alpha}]$ in  $\Bc\big(\Hr^1, \Hr\big)$.
\item $\tilde\mu[L^{1+\alpha}, [d,\partial_L]]L^{-\alpha}P_0\tilde\mu$ 
and $\tilde\mu L^{-\alpha}[L^{1+\alpha},[d,\partial_L]]P_0\tilde\mu$ 
are bounded in $\Hr^s$;
\item $\tilde\mu L[d, \partial_L]P_0\tilde\mu$ and 
$\tilde\mu [d, \partial_L]L P_0\tilde\mu$ are bounded from 
$\Hr^s$ to $\Hr^{s-1}$.
\end{enumerate} 
\end{lemma} 
\proof The first point is obvious.
The two next points follow by noticing that single and double commutators of 
$L$ with $d$ or $\delta_g$ are bounded in $\Hr$. Indeed, we recall
that $L_0$ is the multiplication by 
$r$ above $X_0'$. One derives the result for
$\delta_{\tilde g}$ by conjugating by $U_{\tilde g}$. For the two
last points, it is enough to  take $\alpha=2$ by interpolation and
duality.  A straightforward computation gives that for a smooth
$k$-form $f$  
with compact support on the cusp, we have:
\begin{equation}\label{e:comm1}
\quad [d, \partial_L]P_0 f= \big(2(1-p)x^{1-p}d +
\alpha_k(1-p)x^{2(1-p)}x^{p-2}dx\wedge\big)P_0 f_1,  
\end{equation} 
where $f=f_1+x^{-2}dx\wedge f_2$, following \eqref{decom}. Here, we
used $d P_0=(dx\wedge \partial_x\cdot)P_0$. Hence the fourth point is clear. 
For the last one, note $L
x^{(1-p)}$ is constant on the cusp, for $p< 1$. \qed

\begin{lemma}\label{l:uni}
We set $W_{\rm lr}:= U^{-1}_{\rm lr}\Delta_{\rm lr} U_{\rm lr}$ and $D_{U_{\rm lr}}:=d+U_{\rm lr}\delta U^{-1}_{\rm lr}$  in $\Hr_g$ and we also set $W_{\rm sr}:= U^{-1}_{\rm sr}\Delta_{\rm sr} U_{\rm sr}$ and
$D_{U_{\rm sr}}:=d+U_{\rm sr}\delta U^{-1}_{\rm sr}$ in $\Hr_{g_{\rm lr}}$. Then,
\begin{enumerate}
\item We have $W_{\rm lr}= U_{\rm lr} D_{U_{\rm lr}}^2  U^{-1}_{\rm lr}$ and $W_{\rm sr}= U_{\rm sr} D_{U_{\rm sr}}^2  U^{-1}_{\rm sr}$. Moreover, they are essentially self-adjoint on $\cC_c^{\infty}(X, \Lambda^* X)$ in $\Hr_g$ and $\Hr_{\rm lr}$, respectively. Their domain is $\Hr^2$.
\item We have $(W_{\rm lr} +i)^{-1}-
(\Delta_g+i)^{-1}\in\Kc(\Hr_g)$ and 
$(W_{\rm sr} +i)^{-1}-
(\Delta_{\rm lr}+i)^{-1}\in\Kc(\Hr_{\rm lr})$.
\item The operator $W_{\rm lr}$ is a long-range 
perturbation of $\Delta_g$ inside $\Bc(\Hr^2, \Hr^{-2})$. 
\item The operator $W_{\rm sr}$ is a short-range 
perturbation of $\Delta_{\rm lr}$ inside $\Bc(\Hr^2, \Hr^{-2})$. 
\end{enumerate} 
\end{lemma}
\proof
Since the manifold is complete, $D_{U_{\rm lr}}$ is self-adjoint on the
closure of $\cC^{\infty}_c(X, \Lambda^* X)$ under the norm 
$\|f\|+\|df\|+\|\delta_{g_{\rm lr}} f\|$. Now remark that for all 
$\alpha\in\R$, $(1+\rho_{\rm lr})^\alpha$ stabilizes $\Hr^2$ and by 
duality and interpolation $\Hr^s$ for all $s\in[-2,2]$. Note also that 
$\Dc(U_{\rm lr} D_{\rm lr} U^{-1}_{\rm lr})=\Dc(D)$ to get the first point (the argument being the same for the short-range part).     

We now compare the two operators $F_{lr}:= W_{\rm lr} - \Delta_g $ and $F_{\rm sr}:=W_{\rm sr} - \Delta_{g_{\rm lr}} $ in $\Hr_{g}$ and in $\Hr_{g_{\rm lr}}$, respectively.  We compute on
$\cC^\infty_c(X, \Lambda^* X)$, 
\begin{align*}
F_\cdot=  U_\cdot(D_{U_\cdot}-D)D_{U_\cdot} U_\cdot^{-1} + 
U_\cdot D(D_{U_\cdot}-D)U^{-1}_\cdot+U_\cdot D^2(U_\cdot^{-1}-1)+(U_\cdot-1)D^2.
\end{align*}
Note also that $D_{U_\cdot}-D= (U_\cdot-1)d^* U_\cdot^{-1}- d^*(U_\cdot^{-1}-1)$, where the adjoint of $d$, is taken in $\Hr_{g}$ and in $\Hr_{g_{\rm lr}}$, respectively. 

To get point (2), since $W_{\rm lr}$ and $\Delta_g$ have the same
domain, it is enough to show that $W_{\rm lr}-\Delta_g$ is a compact operator
in $\Bc(\Hr^2, \Hr^{-2})$ by writing the difference
of the resolvent in a generalized way, see \cite[Lemma
6.13]{GMo1}. This follows using $U^{\pm 1}-1\in\Kc(\Hr^1, \Hr)$. The argument is the same of 
the short-range part. 

Focus on point $(4)$. We treat one of the bad terms. 
We compute on $\cC^\infty_c(X, \Lambda^* X)$:
\begin{align*}
L^{1+\varepsilon} \delta_{\rm lr} (U^{-1}_{\rm sr}-1)D_{U_{\rm sr}}U^{-1}_{\rm sr}=&\, \,
\delta_{\rm lr} L^{1+\varepsilon}(U^{-1}_{\rm sr}-1)D_{U_{\rm sr}}U^{-1}_{\rm sr}
\\
&
+ \big([L^{1+\varepsilon}, \delta_{\rm lr}]L^{-\varepsilon}\big)L^{\varepsilon}
(U^{-1}_{\rm sr}-1)D_{U_{\rm sr}}U^{-1}_{\rm sr}. 
\end{align*}
By density, the first term extends to an element of 
$\Bc(\Hr^{1}, \Hr^{-1})$. Indeed, using the invariance of the domains, 
$D_{U_{\rm sr}}U^{-1}_{\rm sr}\in\Bc(\Hr^{1}, \Hr)$ and $L^{1+\varepsilon}(U^{-1}_{\rm sr}-1)\in\Bc(\Hr)$ by Lemma \ref{l:U}. The second term extends to an 
element of $\Bc(\Hr^{1}, \Hr)$ using Lemma \ref{l:L}. The other terms are treated in the same way.

We turn to point $(3)$. We treat as above the term in $L^\varepsilon$.  
It remains to show that $\|L^{1+\varepsilon }\tilde\mu[F_{\rm lr},
\partial_L]P_0\tilde\mu\|_{\Bc\left(\Hr^2, \Hr^{-2}\right)}$ is
finite. We pick a bad term and drop $P_0$, $\tilde\mu$ for clarity.
\begin{equation*}\begin{split}
L^{1+\varepsilon}[\delta (U^{-1}_{\rm lr}-1)D_{U_{\rm lr}} U^{-1}_{\rm lr}, \partial_L]=&\	
L^{1+\varepsilon}[\delta , \partial_L](U^{-1}_{\rm lr}-1)D_{U_{\rm lr}} U^{-1}_{\rm lr}+ L^{1+\varepsilon}\delta [U^{-1}_{\rm lr}, \partial_L] D_{U_{\rm lr}} U^{-1}_{\rm lr}\\
&\hspace*{-1cm}+ L^{1+\varepsilon}\delta (U^{-1}_{\rm lr}-1)[D_{U_{\rm lr}}, \partial_L] U^{-1}_{\rm lr}+L^{1+\varepsilon}\delta (U^{-1}_{\rm lr}-1)D_{U_{\rm lr}}[ U^{-1}_{\rm lr}, \partial_L].  
\end{split}\end{equation*}
After commutation, the first term in the right-hand side becomes 
\[\big([\delta , \partial_L]L\big) 
\big(L^{\varepsilon}(U^{-1}_{\rm lr}-1)\big)D_{U_{\rm lr}} U^{-1}_{\rm lr} +
\big([L^{1+\varepsilon}, [\delta , \partial_L]]L^{-\varepsilon}\big)
 \big(L^\varepsilon(U^{-1}_{\rm lr}-1)\big) D_{U_{\rm lr}} U^{-1}_{\rm lr}\] 
while the second one is
\[\delta \big(L^{1+\varepsilon} [U^{-1}_{\rm lr}, \partial_L]\big)
D_{U_{\rm lr}} U^{-1}_{\rm lr} + \big([L^{1+\varepsilon}, \delta]L^{-\varepsilon}\big) 
\big(L^{\varepsilon}[U^{-1}_{\rm lr}, \partial_L]\big) D_{U_{\rm lr}} U^{-1}_{\rm lr}.\] 
Use Lemma \ref{l:L} and Lemma \ref{l:U} to control the 
terms in brackets. They are elements of $\Bc(\Hr^1, \Hr^{-1})$ by density. 
The other terms are controlled in the same way, by commutation, we let 
$L^{\varepsilon}$ touch $(U^{\pm 1}_{\rm lr}-1)$ and 
$L^{1+\varepsilon}$ touch $[U^{\pm 1}_{\rm lr}, \partial_L]$.\qed

We are now in position to show the Mourre estimate for the perturbed
metric.

\proof[Proof of Theorem \ref{t:Ith1}]
Let $\cJ$ an interval  containing no
element of $\kappa(p)$.
Taking $R$ large enough, Theorem \ref{t:mourre_0} gives $c>0$
and a compact operator $K$ so that the Mourre estimate
\begin{equation}\label{e:T}
E_\cJ(T) [T,iS] E_\cJ(T) \geq c E_\cJ(T)+ K 
\end{equation}
holds in the sense of forms in $L^2(X, \Lambda^k X, g)$, for $T=\Delta_{k}$. 
Moreover, $\Delta_{k}$ belongs to $\cC^{2}\big(S; \Dc(\Delta_{k}), 
L^2(X, \Lambda^k X, g)\big)$.
We dropped  the subscript $R$ of $S$ as it plays no further r\^ole. 

By Lemma \ref{l:uni},  $W_{\rm lr}$ belongs to the class 
$\cC^{1,1}\big(S;\Dc(\Delta_{k}),\Dc(\Delta_{k})^*\big)$. 
In particular, we obtain
$W_{\rm lr}\in\cC^{1}_{\rm u}\big(S;\Dc(\Delta_{k}), \Dc(\Delta_{k})^*\big)$, 
meaning that $t\mapsto e^{itS}W_{\rm lr}e^{-itS}$ is norm continuous in
$\Bc\big(\Dc(\Delta_{k}), \Dc(\Delta_{k})^*\big)$. By the point (2) of Lemma
\ref{l:uni} and \cite[Theorem 7.2.9]{ABG} the inequality \eqref{e:T}
holds for  $T=W_{\rm lr}$ (up to changing $c$ and $K$). Repeat now with the short-range part and get the Mourre estimate with $T=U_{\rm lr}^{-1}W_{\rm sr}U_{\rm lr}$. 

Now go into $\Hr_{\tilde g}$ conjugating with $U_{\tilde g}$. 
The con\-ju\-ga\-te operator becomes
$S_{\tilde g}$. Therefore, $\Delta_{\tilde g, k}\in\cC^{1,1}\big(S_{\tilde g}; \Dc( \Delta_{{\tilde g},k}), \Dc(\Delta_{{\tilde g},k})^*\big)$ and there are $c>0$ and a compact operator $K$ so that 
\begin{equation}\label{e:tildeT}
E_\cJ(T_{\tilde g}) [T_{\tilde g},iS_{\tilde g}] E_\cJ(T_{\tilde g}) \geq c E_\cJ(T_{\tilde g})+ K 
\end{equation}
holds in the sense of forms in $L^2(X, \Lambda^k X, \tilde g)$ for
$T_{\tilde g}=\Delta_{\tilde g, k}$. We now add the perturbation given by
$V_{\rm sr}$ and $V_{\rm lr}$. Note that $H_0:=\Delta_{\tilde g, k} + V_{\rm
  lr}$ has the same domain as $H:= H_0+V_{\rm sr}$ and that
$(H+i)^{-1}- (H_0+i)^{-1}$ is compact by Rellich-Kondrakov lemma. By
Lemma \ref{l:V},  we obtain $H\in\cC^{1,1}\big(S_{\tilde g}; \Dc(H),\Dc(H)^*\big)$. As above, the inequality \eqref{e:tildeT} is true for $\tilde T=H$.   

We now deduce the different claims of the theorem. The first comes from
\cite[Theorem 7.5.2]{ABG}. The second ones is a consequence of the Virial
theorem. For the third point first note that $\Lr_s\subset\Dc(|S|^s)$ for
$s\in [0,2]$ by Lemma \ref{l:sa} and use \cite{ggm}. Finally, the last point follows from \cite[Theorem
 7.6.11]{ABG}. \qed 

We finish this section with two remarks which improve the result.

\begin{remark} Concerning the point \eqref{b)}, we are able to show 
that the eigenvalues in $\kappa(p)$ are of finite multiplicity only 
in the case of the metric \eqref{cume2} i.e.\ when the perturbation is smooth, using
\cite[Lemma B.1]{GMo1}. At every  other energy level, 
it follows from the Mourre estimate \eqref{e:tildeT}, with $\tilde T=H$, via the Virial Lemma. 
\end{remark}

\begin{remark}
If $M$ is disconnected and if one of its connected components $M_0$ has 
Betti numbers $b_k(M_0)=b_{k-1}(M_0)=0$, then by taking $L$ to be $1$
on the corresponding cusp $[0,\infty)\times M_0$, 
any potential with support in this cusp and tending to $0$ at infinity
(without any required speed) is a short-range perturbation, see
\cite{GMo1} for similar statements.
\end{remark} 

\section{Betti numbers and cusps of hyperbolic manifolds}\label{s:betti}

We conclude by some relations between our analysis and the topology of 
finite-volume hyperbolic manifolds. Recall first some definitions and 
basic topological facts. For any integer $k$, the 
Betti number $b_k(M)$ of a smooth manifold $M$ 
is defined as the dimension of the de Rham cohomology
group $H^k(M)$, which equals also the $\qz$-dimension of the rational 
cohomology group $H^k(M,\qz)$. The Betti numbers are always zero outside the 
range $0,\ldots,n$. For instance, the Betti numbers of the sphere
$S^n$ are 1 in dimension $0,n$ and vanish otherwise. 
By the K\"unneth formula, we can compute the Betti numbers of a Cartesian 
product $M\times N$ in terms of the Betti numbers of $M,N$:
\[b_k(M\times N)=\sum_{j=0}^k b_j(M)b_{k-j}(N).\]
Thus the Betti numbers $b_k$ of the torus $T^n:=(S^1)^n$ equal the binomial 
coefficients $\binom{n}{k}$, in particular none of them (in the range 
$0,\ldots,n$) is zero.

It is easy to see that a cylinder $M\times (0,\infty)$, with the metric $g_1$ 
given on the whole cylinder by \eqref{mc} for $p=1$, is
hyperbolic (i.e.\ it has constant sectional curvature $-1$) if and only
if $M$ is flat (i.e.\ its sectional curvatures vanish identically).
Moreover, every finite-volume complete hyperbolic manifold is of this form
outside a compact set \cite{benepetro}. Closed
flat manifolds, also called \emph{Bieberbach manifolds}, have been
classified in  dimension $3$ by Hantzsche and Wendt
\cite{hantzschewendt}. They obtain ten  different topological types of
Bieberbach manifolds, of which  $\mathfrak{A}_1$--$\mathfrak{A}_6$ are
orientable, while $\mathfrak{B}_1$--$\mathfrak{B}_4$ are 
non-orientable. We remark here that from the description of \cite[page
  610]{hantzschewendt},  it follows that $\mathfrak{A}_6$ is the only 
Bieberbach $3$-manifold with first Betti number  equal to $0$. Since
$\mathfrak{A}_6$ is orientable, by Poincar\'e duality we see that
$b_2(\mathfrak{A}_6)=b_1(\mathfrak{A}_6)=0$. Also from loc.\ cit.\ and
from the explicit description of the holonomy groups in \cite{wolf},
the only non-orientable Bieberbach $3$-manifold with second Betti
number $b_2$  equal to $0$ is $\mathfrak{B}_4$, which moreover has
$b_3(\mathfrak{B}_4)=0$ by non-orientability.

As a by-product of our analysis we obtain the following Hodge decomposition 
on asymptotically hyperbolic manifolds.
\begin{proposition}\label{p:hodge}
Let $(X,g)$ be a complete non-compact hyperbolic manifold of finite volume.
Let $M$ be the boundary at infinity. If $n=\dim(X)$ is odd, suppose also that the 
Betti number $b_{\frac{n-1}{2}}(M)=0$. 
Then $\im(d)$ and $\im(\delta_{\tilde g})$ are closed and  
\begin{equation*}
L^2(X, \Lambda^*X)= \ker(\Delta_{\tilde g})\oplus \im(d)\oplus
\im(\delta_{\tilde g}).	 
\end{equation*}
The conclusion holds more generally for any metric
$\tilde g$ which satisfies \eqref{e:asympto} and \eqref{e:asympto2} with respect to 
the given hyperbolic metric.
\end{proposition}
This proposition is a consequence of the fact that $0$ is not 
in the essential spectrum of the Hodge Laplacian 
(compare with \cite{mazphi}) 
and of a result of stability of the essential spectrum 
obtained in \cite{GG}, see Proposition \ref{p:stabess}.
\begin{proof} Every complete noncompact hyperbolic manifold of finite volume 
is of the form \eqref{mc} with $p=1$ outside a compact set.
By Proposition \ref{p:thema}, we obtain that $\Delta$ has $0$ in its essential 
spectrum if and only if $n$ is odd and $b_{(n-1)/2}$ is non-zero.
Therefore, by hypothesis, $0$ is not in the essential spectrum of 
the Hodge Laplacian $\Delta$. Now by Proposition
\ref{p:stabess}, $0$ is not in the essential spectrum of
$\Delta_{\tilde g}$. Using the closed graph theorem, it follows easily
that $\im(\Delta)$ is closed. Since the metric is smooth, one
obtains the closedness of $\im(d)$ and of $\im(\delta_{\tilde g})$
and also the announced Hodge decomposition, see \cite[Theorem 5.10]{bue}.  
\end{proof}

An important geometric question in dimension $4$ is to decide which
flat manifolds occur as cusps of hyperbolic manifolds. It is known that
every manifold in the Hantzsche-Wendt list appears among the cusps
of some finite-volume hyperbolic manifold \cite{niemershiem} but 
it is not known whether
$1$-cusped hyperbolic $4$-manifold exist at all. There is 
an obstruction to the existence of oriented hyperbolic 
$4$-manifolds with certain 
combinations of cusps. This obstruction, discovered by Long and Reid
\cite{longreid1}, is the integrality of the eta invariant 
(for the signature operator) of the oriented 
Bieberbach manifold modeling the cusps. The eta invariant of Dirac
operators on Bieberbach $3$-manifolds was computed by Pf\"affle 
\cite{pfaffle} and may provide additional obstructions.
See also \cite{Gui} and references therein for an introduction to the
eta invariant. 

The essential spectrum of the Laplacian on $k$-forms was computed in
\cite[Theorem 1.11]{mazphi}, in particular it is stated to be
non-empty for all $k$. One implicit assumption in the proof of
\cite[Lemma 5.26]{mazphi}   seems to be the non-vanishing  
of $H^k(M)$; the proof also works provided $H^{k-1}(M)\neq 0$. 
Rafe Mazzeo confirmed to us in a private communication that [24,
Theorem 1.11] holds under the tacit assumption that the manifold $X$
is orientable, and that the cusp cross-section $M$ has non-vanishing 
cohomology in degrees $k$ or $k-1$. There exist however flat $3$-manifolds $M$
for which $b_2(M)=b_{1}(M)=0$, of the type $\mathfrak{A}_6$ in the 
Hantzsche-Wendt classification. By taking Cartesian products of such 
$M$ with itself, we get examples in dimension $3j$ for all $j\geq 1$.
Thus we are led to the following
\begin{Question}\label{q1}
Does there exist a complete finite-volume hyperbolic manifold of dimension $4$
such that each component of its boundary at infinity is the rational homology 
sphere $\mathfrak{A}_6$? 
\end{Question}

As we see in Theorem \ref{t:Ith}, this issue is crucial for 
the nature of the spectrum of the Laplacian on $k$-forms. 
In light of Theorem \ref{t:Ith}, such a manifold would provide a 
counterexample to \cite[Theorem 1.11]{mazphi} 
for $k=2$. If we accepted the result of \cite{mazphi} to continue to hold 
as stated there in full generality, then the answer would always be
negative; however this is unlikely and we conjecture that there exists
indeed such a $4$-manifold.  More generally, we are led to
\begin{Question}
Does there exist a complete finite-volume hyperbolic 
manifold such that
for some $k\in \N$, the boundary at infinity $M$ satisfies 
$b_k(M)=b_{k-1}(M)=0$?
\end{Question}
This question we can answer affirmatively.
In the literature there exists a non-orientable, finite-volume
hyperbolic $4$-manifold with $5$ cusps, all of which are topologically
of type $\mathfrak{B}_4$. It is obtained by gluing together 
the sides of a regular ideal $24$-cell in hyperbolic $4$-space
in a particular way. There are $1171$
such identifications possible \cite{rattsch}; one of them, called here
$\mathfrak{RT}_{1080}$
(\cite[Table 3, nr.\ crit.\ 1080]{rattsch}), has the desired property of having
all its cusps of type $\mathfrak{B}_4$. It follows from Theorem \ref{t:Ith}
that the Laplacians on $3$- and $4$-forms on $\mathfrak{RT}_{1080}$ 
have purely discrete spectra, which shows that the additional hypothesis 
that at least one of $b_k(M),b_{k-1}(M)$ be different from $0$ is necessary for 
\cite[Theorem 1.1]{mazphi} to hold as stated. We remark here that 
Mazzeo-Phillips use the Hodge star operator, which only exists on oriented 
manifolds, to reduce the analysis to $k\leq n/2$, while our counterexample 
is non-orientable and works for $n=4$ and $k\in\{3,4\}$. 

In conclusion, the complete general statement  about the essential
spectrum of the Laplacian on forms on finite-volume hyperbolic
manifolds, without any restriction on the cusp cross-section, is 
obtained from \cite[Theorem 1.11]{mazphi} by adding the
following exceptional  case, which follows from the above
counterexample and our Theorem \ref{t:Ith}: 
\begin{quote}
\emph{In the case where the volume is finite and
all the cusps are modeled on
Bieberbach manifolds with zero Betti numbers in dimensions $k$ and $k-1$,
the spectrum of the Laplacian on $k$-forms is purely discrete and obeys the 
classical Weyl law.}
\end{quote}

The list from \cite{rattsch} includes precisely two hyperbolic 
$4$-manifolds with all cusps of the same type in the Hantzsche-Wendt 
classification. 
One is $\mathfrak{RT}_{1080}$ mentioned above, the other one is
$\mathfrak{RT}_{1011}$, whose five cusps are all of type $\mathfrak{B}_1$.
It is a fact that $\mathfrak{B}_1$ has non-zero Betti numbers $b_1$ 
and $b_2$. 
Except for $\mathfrak{RT}_{1080}$, none of the Rattcliffe-Tschantz
manifolds satisfy $b_2(M)=0$, therefore the corresponding Laplacians 
$\Delta_2$ and $\Delta_3$ have non-empty continuous spectrum 
(Theorem \ref{t:Ith1}). The same is always true for $k=0,1$ since $b_0(M)$ is 
nonzero. However, $\Delta_4$ does have purely discrete spectrum for quite a few 
manifolds from \cite[Tables 3,4]{rattsch}. This is because on one hand, 
$b_4$ of a $3$-manifold is always $0$; on the other hand, $b_3$ of a 
non-orientable $3$-manifold is equally zero. Thus the hypothesis of Theorem 
\ref{t:Ith} is satisfied for $k=4$ for every Rattcliffe-Tschantz
manifold whose cusps are modeled on non-orientable Bieberbach manifolds. Besides
$\mathfrak{RT}_{1080}$ these are:
$\mathfrak{RT}_{35}$, $\mathfrak{RT}_{130}$--$\mathfrak{RT}_{156}$,
$\mathfrak{RT}_{426}$--$\mathfrak{RT}_{516}$, 
$\mathfrak{RT}_{538}$--$\mathfrak{RT}_{544}$
$\mathfrak{RT}_{811}$--$\mathfrak{RT}_{865}$,
$\mathfrak{RT}_{942}$--$\mathfrak{RT}_{959}$, 
$\mathfrak{RT}_{1012}$--$\mathfrak{RT}_{1014}$,
$\mathfrak{RT}_{1080}$--$\mathfrak{RT}_{1084}$,
$\mathfrak{RT}_{1095}$, $\mathfrak{RT}_{1101}$--$\mathfrak{RT}_{1105}$,
$\mathfrak{RT}_{1109}$, $\mathfrak{RT}_{1122}$--$\mathfrak{RT}_{1127}$,
$\mathfrak{RT}_{1134}$, $\mathfrak{RT}_{1142}$, $\mathfrak{RT}_{1143}$,
$\mathfrak{RT}_{1156}$--$\mathfrak{RT}_{1158}$, and $\mathfrak{RT}_{1167}$.
None of the above examples is orientable. Nevertheless,
we expect the answer to Question \ref{q1} to be affirmative.

\appendix
\section{Stability of the essential spectrum}
\renewcommand{\theequation}{A.\arabic{equation}}
\setcounter{equation}{0}
If $A,B$ are closed operators acting in a Hilbert space $\Hr$ and if
there is $z\in\rho(A)\cap\rho(B)$ such that $(A-z)^{-1}-(B-z)^{-1}$
is a compact operator, then we say that {\em $B$ is a compact
  perturbation of $A$}.
If this is the case, then the difference $(A-z)^{-1}-(B-z)^{-1}$ is
a compact operator for all $z\in\rho(A)\cap\rho(B)$. In particular,
{\em if $B$ is a compact perturbation of $A$, then $A$ and $B$ have
  the same essential spectrum}, the essential spectrum of a closed
operator $S$ being the set of complex numbers $\lambda$ such that
$S-\lambda$ is not Fredholm.

For the sake of completeness, we prove the stability of
the essential spectrum of the Laplacian acting on forms under 
perturbations of the metric. As mentioned  in \cite[Remark 9.10]{GG},
this result holds on much weaker hypothesis of regularity of the
metric, see Remark \ref{r:stab}. We stay elementary
and give a proof in the smooth case. Let $(X, g)$ be a connected,
smooth, non-compact Riemannian manifold. Choose $x_0\in X$. Let $d_g$
be the geodesic distance induced by $g$. Let $\Hr:=L^2(X,\Lambda^* X,
g)$ and $\langle\cdot, \cdot\rangle_g$ its scalar product associated to
the metric. Since the metric is complete, the operator $D_g:=d + \delta_g$ is
essentially self-adjoint on $\cC^\infty_c(X, \Lambda^*X)$. Its domain
is $\Gr:=\Dc(d)\cap \Dc(\delta_g)$. The Laplacian in $\Hr$ is $\Delta_g:=D^2$. 

Consider some metric $g'$ on $X$ satisfying 
\begin{align}
\label{e:asympto}
\alpha(x) g_x \leq g'_x \leq \beta(x) g_x,& \mbox{ where
} \alpha(x)  \mbox{ and } \beta(x) \mbox{ tends to } 1, \mbox{ as }
d_g(x, x_0)\rightarrow \infty,
\end{align}
We denote with the subscript $g'$ the object obtained under the metric $g'$.
From \eqref{e:asympto}, we infer there is a unique $a\in\Bc(\Hr)$, $a\geq c>0$,  such that $\langle u, v\rangle_{g'}=\langle u, a v\rangle_g$ for all $u,v\in \Hr$. We now suppose additionally that $a$ is in $\cC^1(D)$,  meaning there is $C$ so that:
\begin{align}
\label{e:asympto2}
\left|\big\langle  Df, a g \big\rangle_g - \big\langle  a f, D g \big\rangle_g \right| \leq C \|f\|_g \|g\|_g , \mbox{ for all } f,g\in \Gr.
\end{align}
Since $a$ is invertible, one obtains easily that $a^{-1}\in\cC^1(D)$ and $a^{\pm 1}\Gr\subset \Gr$.

\begin{proposition}\label{p:stabess}
Under the Assumptions \eqref{e:asympto} and \eqref{e:asympto2}, the Laplace operators $\Delta_g$ and $\Delta_{g'}$, acting on forms, have the same essential spectra. 
\end{proposition} 

\proof Using Assumptions \eqref{e:asympto} and \eqref{e:asympto2} and Rellich-Kondrakov, we get
\begin{align}\label{e:staba}
a^{\pm 1}-1 \in \Kc(\Gr, \Hr).
\end{align} 
Consider $D_{g'}:= d + a^{-1}\delta a$. Thanks to  \eqref{e:staba}, 
it is self-adjoint in $\Hr_{g'}$ with domain $\Gr_{g'}:=\Gr$. Note that
$\Delta_{g'}:=(D_{g'})^2$.  Now remark that the essential spectrum of the closed operator
$\Delta_{g'}$ does not depend on the choice of scalar product on
$\Hr_{g'}$. We look at it in $\Hr$, where it is no longer symmetric.
Therefore, in order to show that
$\Delta_g$ and $\Delta_{g'}$ have the same essential spectrum it is enough 
to show that 
\begin{eqnarray*}
(\Delta_g+1)^{-1}- (\Delta_{g'}+1)^{-1} \in \Kc(\Hr).
\end{eqnarray*} 
Now observe that $(\Delta_\cdot +1)^{-1}= (D_\cdot +i)^{-1}(D_\cdot -i)^{-1}$. Hence, it is enough to show that 
\begin{eqnarray*}
(D_g \pm i)^{-1}- (D_{g'}\pm i)^{-1} \in \Kc(\Hr).
\end{eqnarray*} 
This follows form the second part of \eqref{e:staba}, since:
\begin{align*}
 (D_g\pm i)^{-1}- (D_{g'}\pm i)^{-1}= - (D_g\pm i)^{-1}\big(a^{-1}\delta(a-1)+
  (a^{-1}-1)\delta \big) (D_{g'}\pm i)^{-1} 
\end{align*} 
Indeed, for the first term, we have $(a-1)\in\Kc(\Gr, \Hr)$ and 
$a^{-1}\Gr^*\subset \Gr^*$. Therefore, we get  $a^{-1}\delta(a-1)\in \Kc(\Gr,
\Gr^*)$. For the second one, note that $(a^{-1}-1)\in \Kc(\Hr, \Gr^*)$.\qed
\begin{remark}\label{r:stab0}
The condition \eqref{e:asympto2} follows from the stronger
\begin{align}
\label{e:asympto2'}
\left|\big\langle  d^*f, a g \big\rangle_g - \big\langle  a^* f, d g \big\rangle_g \right| \leq C \|f\|_g \|g\|_g , \mbox{ for all } f\in \Dc(d^*) ,g\in \Dc(d).
\end{align}
In particular for a conformal transformation $g'=(1+\rho)g$ with $d\rho$ in $L^\infty$, Assumption 
\eqref{e:asympto2'} is fulfilled.
Note also that \eqref{e:asympto2'} implies that $a^{\pm 1}$ stabilizes $\Dc(d)$ and $\Dc(\delta)$.
\end{remark} 

\begin{remark}\label{r:stab}
For Laplacians acting on functions, the condition
\eqref{e:asympto2} is automatically fulfilled. Moreover, one may treat measurable
metrics, see \cite[Section 9]{GG}. 
We refer to \cite[Lemma 6.3]{GMo1} for similar results concerning
perturbations of a magnetic field.
\end{remark} 

\bibliographystyle{plain} 
 
\end{document}